\documentclass[12pt,a4paper,reqno]{amsart} 
\usepackage{amssymb}
\usepackage{latexsym}
\usepackage{amsmath}
\usepackage{amsxtra}
\usepackage{mathrsfs}
\usepackage[X2,T1]{fontenc}
\usepackage[applemac]{inputenc}
\usepackage{cite}
\usepackage{calc}                   

\newcommand{\scal}[2]{\langle #1,#2\rangle}
\newcommand{\nn}[1]{\mathbf N^{#1}}
\newcommand{\rr}[1]{\mathbf R^{#1}}

\newcommand{\zz}[1]{\mathbf Z^{#1}}
\newcommand{\nm}[2]{\Vert #1\Vert _{#2}}

\newcommand{\Op}[1]{\operatorname{Op}(#1)}

\newcommand{\sets}[2]{\{ \, #1\, ;\, #2\, \} }

\newcommand{\ep}{\varepsilon}
\newcommand{\fy}{\varphi}
\newcommand{\cdo}{\, \cdot \, }
\newcommand{\supp}{\operatorname{supp}}

\newcommand{\eabs}[1]{\langle #1\rangle}     

\newcommand{\vrum}{\vspace{0.1cm}}

\newcommand{\op}{\operatorname{Op}}
\newcommand{\SG}{\operatorname{SG}}

\newcommand{\norm}[1]{\langle#1\rangle}

\def\cB{{\mathcal B}}

\def\cD{{\mathscr D}}

\def\cF{{\mathscr F}}
\def\cL{{\mathcal L}}

\def\cS{{\mathscr S}}

\def\SG{\operatorname{SG}}

\def\mascP{\mathscr P}

\numberwithin{equation}{section}          

\newtheorem{thm}{Theorem}
\numberwithin{thm}{section}

\newcommand{\rubrik}{}
\newtheorem{prop}[thm]{Proposition}

\newtheorem{lemma}[thm]{Lemma}

\theoremstyle{definition}

\newtheorem{defn}[thm]{Definition}

\theoremstyle{remark}
\newtheorem{rem}[thm]{Remark}

\def\zz#1{\mathbf{Z}^{#1}}

\def\Ph{\mathfrak{F}}
\def\Phr{\mathfrak{F}^r}
\def\binomial#1#2{{#1 \choose #2}}
\def\Div{\mathscr{D}}

\def\Sym#1{\mathrm{Sym}\left(#1\right)}
      
\author{Sandro Coriasco}
\address{Dipartimento di Matematica ``G. Peano'', Università degli
Studi di Torino, Torino, Italy}
\email{sandro.coriasco@unito.it}

\author{Joachim Toft}
\address{Department of Computer science, Physics and Mathematics,
Linn{\ae}us University, V{\"a}xj{\"o}, Sweden}
\email{joachim.toft@lnu.se}

\title{Calculus for Fourier Integral Operators\\in generalized $\mathbf{SG}$ classes}



\begin{document}

\begin{abstract}
We construct a calculus for generalized $\SG$
Fourier integral operators, extending known results to a
broader 
class of symbols of $\SG$ type. In particular, we do 
not require that the phase functions are homogeneous.
%
%
An essential ingredient in the proofs is a general criterion for 
asymptotic expansions within the Weyl-H\"ormander calculus, 
which we previously proved.
%
%
We also prove the $L^2(\rr{d})$-boundedness
of the generalized $\SG$ Fourier integral operators 
having regular phase functions and amplitudes uniformly bounded on $\rr{2d}$. 
\end{abstract}

\maketitle

\setcounter{tocdepth}{1}

\tableofcontents

\section{Introduction}\label{sec0}
%

\par

The aim of this paper is to extend the calculus of Fourier integral
operators based on the so-called $SG$ symbol classes, originally
studied by S. Coriasco \cite{coriasco}, to the more general setting
of generalized $\SG$ symbols introduced in \cite{CJT2} by
S. Coriasco, K. Johansson and J. Toft.

\par

Explicitly, for every $m,\mu \in \mathbf R$,  the standard class
$\SG ^{m,\mu}(\rr{d})$ of $\SG$ symbols, are functions 
$a(x,\xi) \in C^\infty(\rr d\times\rr d)$
with the property
that, for any multiindices $\alpha \in \nn d$ and $\beta \in \nn d$, there exist
constants $C_{\alpha\beta}>0$ such that
\begin{equation}
	\label{disSG}
	|D_x^{\alpha}D_\xi^{\beta} a(x, \xi)| \leq C_{\alpha\beta} 
	\langle x\rangle^{m-|\alpha|}\langle\xi\rangle^{\mu-|\beta|},
	\qquad x, \xi \in \rr d \times \rr d.
\end{equation}
hold.
Here $\langle x \rangle=(1+|x|^2)^{1/2}$ when $x\in\rr d$ and $\mathbf N$
is the set of natural numbers.
These classes together with corresponding
classes of pseudo-differential operators 
$\op (\SG ^{m,\mu})$,
were first introduced in
the '70s by H.O.~Cordes \cite{Co} and C.~Parenti
\cite{PA72}. See also R.~Melrose \cite{Me}. They form a
graded algebra, i.{\,}e.,
$$
\op (\SG ^{m_1,\mu _1})\circ \op (\SG ^{m_2,\mu _2})
\subseteq \op (\SG ^{m_1+m_2,\mu _1+\mu _2}),
$$
whose residual
elements are operators with symbols in
\[
	 \SG ^{-\infty,-\infty}(\rr d)= \bigcap_{(m,\mu) \in \rr 2} \SG ^{m,\mu} (\rr d)
	 =\cS(\rr {2d}),
\]
that is, those having kernel in $\cS(\rr {2d})$, continuously
mapping  $\cS^\prime(\rr d)$ to $\cS(\rr d)$.

\par

Operators in $\op (\SG ^{m,\mu})$
are continuous on $\cS(\rr d)$, and extend uniquely to
continuous operators on $\cS^\prime(\rr d)$ and from
$H^{s,\sigma}(\rr d)$ to $H^{s-m,\sigma-\mu}(\rr d)$.
Here $H^{t,\tau}(\rr d)$,
$t,\tau \in \mathbf R$, denotes the weighted Sobolev space
\begin{equation*}
  	H^{t,\tau}(\rr d)= \{u \in \cS^\prime(\rr {d}) \colon \|u\|_{t,\tau}=
	\|\norm x^t\norm D^\tau u\|_{L^2}< \infty\},
\end{equation*}

\par

An operator $A=\Op{a}$, is called \emph{elliptic}
(or $\SG ^{m,\mu}$-\emph{elliptic}) if
$a\in \SG ^{m,\mu} (\rr d)$ and there
exists $R\ge0$ such that
%
\[
	C\norm{x}^{m} \norm{\xi}^{\mu}\le |a(x,\xi)|,\qquad 
	|x|+|\xi|\ge R,
\] 
for some constant $C>0$.

\par

An elliptic $\SG $ operator $A \in \op (\SG ^{m,\mu})$ admits a
parametrix $P\in \op (\SG ^{-m,-\mu})$ such that
\[
PA=I + K_1, \quad AP= I+ K_2,
\]
for suitable $K_1, K_2$, smoothing operators with symbols in
$\SG^{-\infty,-\infty}(\rr{d})$, and it turns out to be a Fredholm
operator on the scale of functional spaces $H^{t,\tau}(\rr d)$,
$t,\tau\in\mathbf R$.

\par

In 1987, E.~Schrohe \cite{Sc} introduced a class of non-compact
manifolds, the so-called $\SG$ manifolds, on which it is possible
to transfer from $\rr d$ the whole $\SG$ calculus.
These are manifolds which admit a finite atlas whose changes of
coordinates behave like symbols of order $(0,1)$ (see \cite{Sc}
for details and additional technical hypotheses). An especially
interesting example of  $\SG$ manifolds are the manifolds with
cylindrical ends, where also the concept of classical $\SG$
operator makes sense,
see, e.{\,}g. \cite{BaCo,CoSc14,CoMa2,ES97,MP02,Me}.

\par

The calculus of corresponding classes of Fourier integral
operators, in the forms
\begin{align*}
f\mapsto (\op _\fy (a)f)(x)&= (2\pi )^{-d}\int _{\rr d}e^{i\fy (x,\xi )} a(x,\xi )
\widehat f(\xi )\, d\xi ,
\intertext{and}
f\mapsto (\op _\fy (a)^*f)(x)&= (2\pi )^{-d}\iint _{\rr {2d}}
e^{i(\scal x\xi -\fy (y,\xi ))} \overline {a(y,\xi )} f(y)\, dyd\xi,
\end{align*}
$f\in\cS(\rr{d})$,
started in \cite{coriasco}. Here the operators $\op _\fy (a)$ and
$\op _\fy ^*(a)=\op _\fy (a)^*$ are called Fourier integral operators
of type I and type II, respectively, with amplitude $a$ and
phase function $\fy$. Note that the type II operator
$\op_\fy(a)^*$ is the formal $L^2$-adjoint of the 
type I operator $\op_\fy(a)$.

We assume that the phase function $\fy$ belongs
$\SG ^{1,1}(\rr d)$ and satisfy
\begin{equation}\label{phasecondCones}
\eabs {\fy ' _x(x,\xi )}\asymp \eabs \xi \quad \text{and}\quad
\eabs {\fy ' _\xi (x,\xi )}\asymp \eabs x,
\end{equation}
if nothing else is stated. Here and in what follows, $A\asymp B$
means that $A\lesssim B$ and $B\lesssim A$,
where $A\lesssim B$ means that
$A\le c\cdot B$, for a suitable constant $c>0$.
In many cases, especially when studying the
mapping properties of such operators,
$\fy$ should also fulfill the usual (global)
non-degeneracy condition
$$
|\det (\fy '' _{x\xi}(x,\xi ))|\ge c ,\qquad x,\xi \in \rr d,
$$
for some constant $c>0$. 
The calculus developed in \cite{coriasco} has been first applied to
the analysis of the well-posedness, in the scale of weighted
spaces $H^{t,\tau}(\rr{d})$, of certain hyperbolic Cauchy problems.
These involved linear operators whose coefficients have, at most, 
polynomial growth at infinity, and was studied in \cite{coriasco2}.
 
 \par

The analysis of such Fourier integral operators subsequently
developed into an interesting 
field of active research, with extensions in many different directions.
For example, an approach involving more
general phase functions compared to \cite{coriasco, coriasco2}
can be found in \cite{AndPhD} by Andrews. In \cite{CaRo}
Cappiello and Rodino deduce results involving Gelfand-Shilov
spaces, and in \cite{CNR07, CorNicRod1}, boundedness
on $\mathcal{F}{L^p(\rr d)} _\mathrm{comp}$ and the
modulation spaces are obtained.
Furthermore, these results are applied in \cite{CoMa2}
to obtain Weyl formulae for the asymptotic behavior of the counting
function for elliptic self-adjoint operators of $\SG$ type, with positive
orders, on manifolds with ends, through (a variant of) the so-called
stationary phase method. The $L^p(\rr{d})$-continuity
of the above operators is studied in \cite{CoRu}, extending
to the global $\rr d$ situation a celebrated result by Seeger,
Sogge and Stein in \cite{SSS91}, valid on compact manifolds.

\par

More general $\SG$ symbol classes, denoted by
$\SG^{(\omega)}_{r,\rho}(\rr d)$, $r,\rho\ge0$, $r+\rho>0$,
have been introduced in the aforementioned paper \cite{CJT2}.
In place of the estimates \eqref{disSG}, $a\in\SG^{(\omega)}
_{r,\rho}(\rr d)$ satisfies
\begin{equation}\label{eq:genSG}
		|D_x^{\alpha}D_\xi^{\beta} a(x, \xi)| \leq C_{\alpha\beta} 
		\omega(x,\xi)\langle x\rangle^{-r|\alpha|}\langle\xi\rangle
		^{-\rho|\beta|}
\end{equation}
for suitable weight $\omega$ and constants $C_{\alpha\beta}>0$,
see Subsections \ref{subs:1.1} and \ref{subs:1.3} below. For the
corresponding pseudo-differential operators, continuity results
and the propagation of singularities, in terms of global wave-front
sets are established in \cite{CJT2,CJT3}. (See also
\cite{Me,CoMa} for related results.) These generalized
$\SG$ symbol classes are well suited when investigating
singularities in the context of modulation and
Fourier-Lebesgue spaces. (See
\cite{F1,Feichtinger3, Feichtinger6} for details on these
functional spaces.)

\medspace

\par

In Section \ref{sec2} we extend the calculus developed in \cite{coriasco}
to include operators $\op_\fy(a)$ and $\op_\fy^*(a)$ with phase functions
in $\SG^{1,1}_{1,1}(\rr{d})$ and amplitudes in the 
generalized $\SG$ classes \eqref{eq:genSG}. 
More precisely, in the first part of Section \ref{sec2} we prove that 
for every $a,b\in \SG ^{(\omega _0)}$ and $p\in \SG ^{(\omega )}$
we have
\begin{align*}
\op (p)\circ \op _\fy (a) &= \op _\fy (c_1) \mod \op (\cB _0),
\\[1ex]
\op (p)\circ \op _\fy ^*(b) &= \op _\fy ^*(c_2) \mod \op (\cB _0),
\\[1ex]
\op _\fy (a) \circ \op (p) &= \op _\fy (c_3) \mod \op (\cB _0)
\\[1ex]
\op _\fy ^*(b) \circ \op (p) &= \op _\fy ^*(c_4) \mod \op (\cB _0),
\end{align*}
for some $c_j\in \SG ^{(\omega_{0,j})}$, $j=1,\dots ,4$, and suitable weights
$\omega_{0,j}$. Here $\op (\cB _0)$
is a set of appropriate smoothing operators, depending on the 
symbols and the phase function. Furthermore,
if $a\in \SG ^{(\omega _1)}$ and $b\in \SG ^{(\omega _2)}$, then it is also
proved that $\op _\fy ^*(b)\circ \op _\fy (a)$ and $\op _\fy(a)\circ \op _\fy^*
(b)$ are equal to pseudo-differential operators $\op (c_5)$ and $\op (c_6)$,
respectively, for some $c_5,c_6\in \SG ^{(\omega _{0,j})}$, $j=5,6$.
We also present asymptotic formulae for $c_j$, $j=1,\dots ,6$,
in terms of $a$ and $b$, or of
$a$, $b$ and $p$, modulo smoothing terms,
with symbol which in most cases
belong to
$\SG^{-\infty,-\infty} = \mathscr S$. %

\par

The results shown in this paper are an essential part of  
the study of the propagation of singularities,
in the context of general modulation spaces, from the data 
to the solutions of the Cauchy problems considered in \cite{coriasco2,CoMa}. 
Another application of the calculus developed here
has been the proof of boundedness results between
suitable couples of weighted modulation spaces for
the class of Fourier integral operators studied here. Both these applications
are examined in \cite{CJT4}.

\par

The paper is organized as follows. In Section \ref{sec1} we recall the
needed definitions and some basic results concerning the generalized
$\SG$ symbol classes. In Section \ref{sec2} we give the definition
of the generalized $\SG$ Fourier integral operators, and prove our main 
results, i.{\,}e., the composition theorems between generalized pseudo-differential
and generalized Fourier integral operators of $\SG$ type, as well as
between the Fourier integral operators. The parametrices for the elliptic elements
are also studied, together with an adapted version of the Egorov theorem.
%
%
In Section \ref{sec3} we discuss the global $L^2(\rr d)$-boundedness of
the generalized $\SG$ Fourier integral operators under the hypotheses 
that the phase function is regular, see Subsection \ref{subs:2.1} below, and the
amplitude is uniformly bounded on $\rr{2d}$.
%
%

\par

\section{Preliminaries}\label{sec1}
%

\par

We begin by fixing the notation and recalling some basic concepts
which will be needed below.
In Subsections
\ref{subs:1.1}-\ref{subs:1.3} we mainly summarizes part of the contents
of Sections 2 in \cite{CJT2}, and
in \cite{CJT3}. In Subsection
\ref{subs:1.4} we
state a few lemmas which
will be useful in the subsequent Section \ref{sec2}. Some of these,
compared with their original formulation in the $\SG$ context,
appeared in \cite{coriasco}, are here given in a more general
form, adapted to the definitions given in Subsection \ref{subs:1.3}.

\par
\subsection{Weight functions}\label{subs:1.1}
Let $\omega$ and $v$ be positive measurable functions
on $\rr d$. Then $\omega$ is called $v$-moderate if
\begin{equation}\label{moderate}
\omega (x+y) \lesssim \omega (x)v(y)
\end{equation}
If $v$ in \eqref{moderate} can be chosen as a polynomial, then $\omega$ is
called a function or weight of \emph{polynomial type}.
We let $\mathscr P(\rr d)$ be the set
of all polynomial type functions on $\rr d$. If $\omega (x,\xi
)\in \mathscr P(\rr {2d})$ is constant with respect to the
$x$-variable or the $\xi$-variable, then we sometimes write $\omega
(\xi )$, respectively $\omega (x)$, instead of $\omega (x,\xi )$,
and consider
$\omega$ as an element in $\mathscr P(\rr {2d})$ or in $\mathscr P(\rr
d)$ depending on the situation. We say that $v$ is submultiplicative
if \eqref{moderate} holds for $\omega=v$. For convenience we assume
that all submultiplicative weights are even, and
$v$ and $v_j$ always stand for submultiplicative weights, if nothing else is stated.

\par

Without loss of generality we may assume that every $\omega \in
\mathscr P(\rr d)$ is smooth and satisfies the ellipticity condition
$\partial ^\alpha \omega / \omega \in L^\infty$. In fact,
by Lemma 1.2 in \cite {To8} it follows
that for each $\omega \in \mathscr P(\rr d)$, there is a smooth and
elliptic $\omega _0\in \mathscr P(\rr d)$ which is equivalent to
$\omega$ in the sense
\begin{equation}\label{onestar}
\omega \asymp \omega _0.
\end{equation}

\par

The weights involved in the sequel have to satisfy additional conditions.
More precisely let $r,\rho \ge0$. Then $\mathscr P_{r,\rho}(\rr {2d})$ is
the set of all $\omega (x,\xi )$ in $\mathscr P(\rr {2d})\bigcap$ $C^\infty (\rr
{2d})$ such that
\begin{equation}\label{SGCondWeight}
\eabs x^{r|\alpha |}\eabs \xi ^{\rho |\beta |}\frac {\partial ^\alpha
_x\partial ^\beta _\xi \omega (x,\xi )}{\omega (x,\xi )}\in L^\infty
(\rr {2d}),
\end{equation}
for every multi-indices $\alpha$ and $\beta$. Any weight
$\omega\in\mascP_{r,\rho}(\rr {2d})$ is then called $\SG$ moderate on
$\rr{2d}$, of order $r$ and $\rho$. Note that $\mathscr P_{r,\rho}$
is different here compared to \cite{CJT1}, and
there are elements in $ \mathscr P(\rr {2d})$ which have
no equivalent elements in $\mathscr P_{r,\rho}(\rr {2d})$.
On the other hand, if $s, t\in
\mathbf R$ and $r, \rho \in [0,1]$, then $\mathscr P_{r,\rho} (\rr {2d})$
contains all weights of the form
\begin{equation}\label{varthetaWeights}
\vartheta _{m,\mu} (x,\xi )\equiv \eabs x ^m\eabs \xi ^\mu ,
\end{equation}
which are one of the most common type of weights.

\par

It will also be useful to consider $\SG$ moderate weights in
one or three sets of variables. Let $\omega \in \mathscr
P(\rr {3d})\bigcap C^\infty (\rr {3d})$, and let $r_1,r_2,\rho
\ge 0$. Then $\omega$ is called $\SG$ moderate on $\rr{3d}$,
of order $r_1$, $r_2$ and $\rho$, if  it fulfills
$$
\eabs {x_1}^{r_1|\alpha _1|} \eabs {x_2}^{r_2|\alpha _2|} \eabs
\xi ^{\rho |\beta |} \frac{\partial _{x_1}^{\alpha _1} \partial _{x_2}
^{\alpha _2} \partial _\xi ^\beta \omega (x_1,x_2,\xi )}
       {\omega (x_1,x_2,\xi )}
\in L^\infty(\rr {3d}).
$$
The set of all $\SG$ moderate weights on $\rr{3d}$ of order
$r_1$, $r_2$ and $\rho$ is denoted by $\mascP _{r_1,r_2,
\rho}(\rr{3d})$. Finally, we denote by $\mascP_{r}(\rr{d})$ the
set of all $\SG$ moderate weights of order $r\ge0$ on $\rr d$,
which are defined in a 
similar fashion.

\par

\subsection{Pseudo-differential operators and generalized $\SG$
symbol classes}\label{subs:1.3}
%
Let $a\in
\mathscr S(\rr {2d})$, and $t\in \mathbf R$ be fixed. Then
the pseudo-differential operator $\op _t(a)$ is the linear and
continuous operator on $\mathscr S(\rr d)$ defined by the formula
\begin{equation}\label{e0.5}
(\op _t(a)f)(x)
=
(2\pi ) ^{-d}\iint e^{i\scal {x-y}\xi } a((1-t)x+ty,\xi )f(y)\,
dyd\xi 
\end{equation}
(cf. Chapter XVIII in \cite {Ho1}). For
general $a\in \mathscr S'(\rr {2d})$, the pseudo-differential
operator $\op _t(a)$ is defined as the continuous operator from
$\mathscr S(\rr d)$ to $\mathscr S'(\rr d)$ with distribution
kernel
\begin{equation}\label{weylkernel}
K_{t,a}(x,y)=(2\pi )^{-d/2}(\mathscr F_2^{-1}a)((1-t)x+ty,x-y).
\end{equation}
Here and in what follows, $\cF f =\widehat f$ is the Fourier transform
of $f\in \cS '(\rr d)$ which takes the form
$$
(\cF f)(\xi )= \widehat f(\xi ) = (2\pi )^{-d/2}\int _{\rr d}f(x)e^{-i\scal x\xi }\, dx
$$
when $f\in \cS (\rr d)$, and $\mathscr F_2F$ is the partial Fourier
transform of $(x,y)\mapsto F(x,y)$ with respect to the $y$-variable.
%
%
%
%

\par

If  $t=0$, then $\op _t(a)$ is the Kohn-Nirenberg
representation $\op (a)=a(x,D)$, and if $t=1/2$, then $\op _t(a)$ is
the Weyl quantization. 

\par

In most of our situations, $a$ belongs to a generalized $\SG$ symbol
class, which we shall consider now.
Let $m,\mu ,r, \rho \in \mathbf R$
be fixed. Then the $\SG$ class $\SG
^{m,\mu }_{r,\rho}(\rr {2d})$ is the set of all $a\in
C^\infty (\rr {2d})$ such that
$$
|D _x^\alpha D _\xi ^\beta a(x,\xi )|\lesssim
\eabs x^{m-r|\alpha|}\eabs \xi ^{\mu -\rho |\beta
|},
$$
for all multi-indices $\alpha$ and $\beta$. Usually we assume that
$r,\rho \ge 0$ and $\rho +r >0$.

\par

More generally, assume that $\omega \in \mathscr P_{r ,\rho}
(\rr {2d})$. Then $\SG _{r,\rho}^{(\omega )}(\rr {2d})$ consists of all $a\in
C^\infty (\rr {2d})$ such that
\begin{equation}\label{Somegadef}
	|D_x^\alpha D_\xi ^\beta a(x,\xi)|\lesssim \omega(x,\xi)\eabs
        x^{-r|\alpha|}\eabs \xi ^{-\rho|\beta|}, \qquad
        x,\xi\in \rr d,
\end{equation}
for all multi-indices $\alpha$ and $\beta$. We note that
\begin{equation}\label{SG}
\SG _{r,\rho}^{(\omega )}(\rr {2d})=S(\omega ,g_{r,\rho}
),
\end{equation}
when $g=g_{r,\rho}$ is the Riemannian metric on $\rr {2d}$,
defined by the formula
\begin{equation}\label{riemannianmetric}
\big (g_{r,\rho}\big )_{(y,\eta )}(x,\xi ) =\eabs y ^{-2r}|x|^2 +\eabs
\eta ^{-2\rho}|\xi |^2
\end{equation}
(cf. Section 18.4--18.6 in \cite{Ho1}). Furthermore, $\SG ^{(\omega
)}_{r,\rho} =\SG ^{m,\mu }_{r,\rho}$ when $\omega$ coincides with
the weight  $\vartheta _{m,\mu}$ defined in \eqref{varthetaWeights}.

\par

For conveniency we set
\begin{gather*}
\SG ^{(\omega \vartheta _{-\infty ,0} )} _\rho (\rr {2d})
= \SG ^{(\omega \vartheta _{-\infty ,0} )} _{r,\rho} (\rr {2d})
\equiv
\bigcap _{N\ge 0} \SG ^{(\omega \vartheta _{-N ,0} )} _{r,\rho} (\rr {2d}),
\\[1ex]
\SG ^{(\omega \vartheta _{0,-\infty } )} _r (\rr {2d})
= \SG ^{(\omega \vartheta _{0,-\infty } )} _{r,\rho} (\rr {2d})
\equiv
\bigcap _{N\ge 0} \SG ^{(\omega \vartheta _{0,-N } )} _{r,\rho} (\rr {2d}),
\intertext{and}
\SG ^{(\omega \vartheta _{-\infty ,-\infty } )}  (\rr {2d})
= \SG ^{(\omega \vartheta _{-\infty ,-\infty } )} _{r,\rho} (\rr {2d})
\equiv
\bigcap _{N\ge 0} \SG ^{(\omega \vartheta _{-N ,-N} )} _{r,\rho} (\rr {2d}).
\end{gather*}
We observe that $\SG ^{(\omega \vartheta _{-\infty ,0} )} _{r,\rho} (\rr {2d})$
is independent of $r$, $\SG ^{(\omega \vartheta _{0,-\infty } )} _{r,\rho} (\rr {2d})$
is independent of $\rho$, and that
$\SG ^{(\omega \vartheta _{-\infty ,-\infty } )} _{r,\rho} (\rr {2d})$ is independent
of both $r$ and $\rho$. Furthermore, for any $x_0,\xi _0\in \rr d$ we have
\begin{alignat*}{3}
\SG ^{(\omega \vartheta _{-\infty ,0} )} _\rho (\rr {2d}) &=
\SG ^{(\omega _0\vartheta _{-\infty ,0} )} _\rho (\rr {2d}),&
\quad &\text{when}& \quad \omega _0(\xi ) &= \omega (x_0,\xi ),
\\[1ex]
\SG ^{(\omega \vartheta _{0,-\infty } )} _r (\rr {2d}) &=
\SG ^{(\omega _0\vartheta _{0,-\infty } )} _r (\rr {2d}),&
\quad &\text{when}& \quad \omega _0(x ) &= \omega (x,\xi _0),
\intertext{and}
\SG ^{(\omega \vartheta _{-\infty ,-\infty } )}  (\rr {2d}) &=
\mathscr S(\rr {2d}).& & & &
\end{alignat*}

%
%
%
%
%
%
%
%
%
%

\par

The following result shows that
the concept of asymptotic expansion extends to the classes
$\SG^{(\omega)}_{r,\rho}(\rr{2d})$. We refer to
\cite[Theorem 8]{CoTo} for the proof.

\par

\begin{prop}\label{propasymp}
Let $r,\rho \ge 0$ satisfy $r+\rho >0$, and let
$\{ s_j \} _{j\ge 0}$ and $\{ \sigma _j \} _{j\ge 0}$ be sequences of
non-positive numbers
such that $\lim _{j\to \infty} s_j =-\infty$ when $r>0$ and $s_j=0$ otherwise,
and $\lim _{j\to \infty} \sigma _j =-\infty$ when $\rho >0$ and $\sigma _j=0$
otherwise. Also let $a_j\in\SG ^{(\omega_j)}_{r,\rho}(\rr{2d})$, $j=0,1,\dots$,
where $\omega_j=\omega \cdot \vartheta_{s_j,\sigma_j}$. Then there is a
symbol $a\in\SG^{(\omega)}_{r,\rho}(\rr{2d})$ such that
\begin{equation}
	\label{eq:sgassum}
	a-\sum_{j=0}^Na_j\in\SG^{(\omega_{N+1})}_{r,\rho}(\rr{2d}).
\end{equation}
The symbol $a$ is uniquely determined modulo a remainder $h$, where
\begin{equation}
	\label{eq:gensgasexp}
	\begin{alignedat}{3}
		h &\in \SG ^{\omega \vartheta _{-\infty ,0})}_{\rho}(\rr{2d}) & \quad
		&\text{when}& \quad r&>0,
		\\[1ex]
		h &\in \SG^{(\omega \vartheta _{0,-\infty} )}_{r}(\rr{2d}) & \quad
		&\text{when} & \quad \rho &> 0,
		\\[1ex]  
		h &\in \mathscr S(\rr{2d}) & \quad &\text{when} &\quad r &> 0,
		\rho >0.
	\end{alignedat}
\end{equation}
\end{prop}

\par

\begin{defn}\label{def:gensgasexp}
The notation $a\sim \sum a_j$ is used when $a$ and $a_j$
fulfill the hypothesis in Proposition \ref{propasymp}. Furthermore,
the formal sum
$$
\sum_{j\ge0}a_j
$$
is called (\emph{generalized $\SG$}) \emph{asymptotic expansion}.
\end{defn}

\par

It is a well-known fact that $\SG$ operators give rise to linear continuous mappings 
from $\mathscr S(\rr d)$ to itself, extendable as linear continuous mappings
from $\mathscr S'(\rr d)$ to itself. They also act continuously between general
weighted modulation spaces, see \cite{CJT2}. 

\par

\subsection{Composition and further properties of $\SG$ classes of
symbols, amplitudes, and functions}\label{subs:1.4}
We define families of \emph{smooth functions with $\SG$ behaviour}, depending on
one, two or three sets of real variables (cfr. also \cite{CoSch}). We then introduce
pseudo-differential operators defined by means of $\SG$ amplitudes. Subsequently,
we recall sufficient conditions for maps of $\rr{d}$ into itself
to keep the invariance of the $\SG$ classes. 

%
%

\par

In analogy of $\SG$ amplitudes defined on $\rr {2d}$, we consider
corresponding classes of amplitudes defined on $\rr {3d}$. More precisely,
for any $m_1, m_2, \mu, r_1,r_2,\rho \in \mathbf R$,
let $\SG ^{m_1,m_2,\mu }_{r_1,r_2,\rho} (\rr{3n})$ be the set of all
$a \in C^\infty \left( \rr{3d} \right )$ such that
%
%
\begin{equation}
\label{eq:1.1}       
                  |\partial ^{\alpha _1}_{x_1}
                  \partial ^{\alpha _2}_{x_2}
                   \partial _\xi ^\beta
                  a(x_1, x_2, \xi)|               
                 \lesssim
                   \eabs{x_1}^{m_1 - r_1|\alpha _1|}
                  \eabs{x_2}^{m_2 - r_2|\alpha _2|} \eabs{\xi}^{\mu  - \rho |\beta |},
\end{equation}
for every multi-indices $ \alpha_1, \alpha_2, \beta$. We usually assume
$r_1,r_2,\rho \ge 0$ and $r_1+r_2+\rho>0$. More generally,  let
$\omega \in \mascP _{r_1,r_2,\rho}(\rr{3d})$. Then
$\SG ^{(\omega )}_{r_1,r_2,\rho} (\rr{3d})$ is the  set of all
$a \in C^\infty \left( \rr{3d} \right)$ which satisfy
\begin{equation}
\tag*{(\ref{eq:1.1})$'$}
                   |\partial ^{\alpha _1}_{x_1}
                  \partial ^{\alpha _2}_{x_2}
                   \partial _\xi ^\beta
                  a(x, y, \xi)|               
                 \lesssim
                   \omega(x_1,x_2,\xi )\eabs{x_1}^{- r_1|\alpha _1|}
                  \eabs{x_2}^{- r_2|\alpha _2|} \eabs{\xi}^{- \rho |\beta |},
\end{equation}
for every multi-indices $ \alpha_1, \alpha_2, \beta$. The set
$\SG ^{ (\omega) }_{r_1,r_2,\rho} (\rr{3n})$ is equipped with
the usual Fr\'{e}chet topology based upon the seminorms
implicit in \eqref{eq:1.1}$'$. 

\par

As above,
$$
\SG ^{(\omega)}_{r_1,r_2,\rho} =\SG ^{m_1,m_2,\mu }_{r_1,r_2,\rho}
\quad \text{when}\quad
\omega (x_1,x_2,\xi )=\eabs
{x_1}^{m_1}\eabs {x_2}^{m_2}\eabs \xi ^\mu .
$$
%

\par

\begin{defn}
\label{def:psidoamp}
Let $r_1,r_2,\rho\ge0$, $r_1+r_2+\rho>0$, and
let $a\in \SG^{(\omega)}_{r_1,r_2,\rho}(\rr{3d})$, where $\omega \in
\mascP_{r_1,r_2,\rho}(\rr{3d})$. Then, the pseudo-differential operator
$\op(a)$ is the linear and continuous operator from $\cS(\rr d)$ to
$\cS^\prime(\rr d)$ with distribution kernel
\[
	K_{a}(x,y)=(2\pi )^{-d/2}(\mathscr F_3^{-1}a)(x,y,x-y).
\]
For $f\in\cS(\rr d)$, we have
\[
	(\op(a)f)(x)=(2\pi)^{-d}\iint e^{i \scal {x-y}\xi} a(x,y,\xi)f(y)\,dyd\xi.
\]
\end{defn}
The operators introduced in Definition \ref{def:psidoamp} have properties
analogous to the usual $\SG$ operator families described in \cite{Co}.
They coincide with the operators defined in the previous subsection,
where corresponding symbols are obtained by means of asymptotic
expansions, modulo remainders of
the type given in \eqref{def:gensgasexp}.
For the sake of brevity, we here omit the details. Evidently, when
neither the amplitude functions $a$, nor the corresponding weight $\omega$,
depend on $x_2$, we obtain the definition of
$\SG$ symbols and pseudo-differential operators, given in the
previous subsection.

\par

Next we consider \emph{$\SG$ functions}, also called
\emph{functions with $\SG$ behavior}. That is, amplitudes which
depend only on one set of variables in $\rr d$. We denote them
by $\SG ^{(\omega )}_r(\rr d)$ and $\SG^{m}_r(\rr d)$, $r>0$,
respectively, for a general weight $\omega\in\mascP_{r}(\rr{d})$
and for $\omega(x)=\eabs x^m$. Furthermore, if
$\phi \colon \rr {d_1} \to \rr {d_2}$,
and each component $\phi _j$, $j=1, \dots, d_2$, of $\phi$ belongs to
$\SG^{(\omega)}_r(\rr {d_1})$, we will occasionally write $\phi \in
\SG^{(\omega)} _r(\rr {d_1};\rr {d_2})$. We use similar notation
also for other vector-valued $\SG$ symbols and amplitudes.

\par

In the sequel we will need to consider compositions of $\SG$
amplitudes with functions with $\SG$ behavior. In particular,
the latter will often be $\SG$ maps (or diffeomorphisms) with
$\SG^0$-parameter dependence, generated by phase
functions (introduced in \cite{coriasco}), 
see Definitions \ref{def:sgmap} and \ref{def:sgmap}, and
Subsection \ref{subs:2.1} below. For the convenience of the
reader, we first recall, in a form slightly more general than the 
one adopted in \cite{coriasco}, the definition $\SG$
diffeomorphisms with $\SG^0$-parameter dependence.

\par

\begin{defn}\label{def:sgmap}
Let $\Omega _j \subseteq \rr{d_j}$ be open, $\Omega =
\Omega _1\times \cdots \times \Omega _k$ and let
$\phi \in C^\infty (\rr d\times \Omega ;\rr d)$. Then $\phi$ is
called an $\SG$ map (with $\SG
^0$-parameter dependence) when the following conditions hold:
\begin{enumerate}
	\item \label{cond:1}
	$\norm{\phi(x,\eta )}\asymp \norm {x}$, 
	uniformly with respect to $\eta \in \Omega$;
	\item for all $\alpha \in \zz {d}_+$, $\beta =
	(\beta _1,\dots ,\beta _k)$, $\beta _j \in \zz{d_j}_+$,
	$j=1,\dots, k$,
	and any $(x,\eta )\in \rr{d} \times \Omega$,
	\[
		|\partial ^\alpha _x\partial ^{\beta _1}_{\eta _1}\cdots  \partial
		^{\beta _k}_{\eta _k}\phi(x,\eta )|
		\lesssim
		\norm{x}^{1-|\alpha |}\norm{\eta _1}^{-|\beta _1|}\cdots
		\norm{\eta _k}^{-|\beta _k|},
	\]
where $\eta =(\eta _1,\dots ,\eta _k)$ and $\eta _j\in \Omega _j$ for every
$j$.
\end{enumerate}
\end{defn}
%
%
%
%

\par

\begin{defn}\label{def:sgdiffeo}
Let $\phi \in C^\infty (\rr{d}\times \Omega ;\rr d )$ be an $\SG$
map. Then $\phi$ is called an $\SG$
diffeomorphism (with $\SG^0$-parameter dependence) when
there is a constant $\varepsilon>0$ such that
\begin{equation}\label{eq:sgreg}
	|\det \phi ^\prime _x(x,\eta )|\ge\varepsilon,
\end{equation}
uniformly with respect to $\eta \in \Omega$.
\end{defn}
\begin{rem}
	Condition (\ref{cond:1}) in Definition \ref{def:sgmap} and
	\eqref{eq:sgreg}, together with abstract results
	(see, e.g., \cite{Berger}, page 221) and the inverse function
	theorem, imply that,
	for any $\eta \in \Omega$, an $\SG$ diffeomorphism
	$\phi(\cdo ,\eta )$ is a smooth, global bijection from
	$\rr {d}$ to itself with smooth inverse $\psi(\cdo ,\eta )
	=\phi ^{-1}(\cdo ,\eta )$.
	It can be proved that also the inverse mapping
	$\psi (y,\eta )=\phi ^{-1}(y,\eta )$ fulfills Conditions
	(1) and (2) in Definition \ref{def:sgmap}, 
	as well as \eqref{eq:sgreg}, see \cite{coriasco}.
\end{rem}

\par

\begin{defn}\label{def:omegainv}
	Let $r,\rho \ge 0$, $r+\rho >0$, $\omega \in \mascP_{r,\rho}
	(\rr{2d})$, and let $\phi ,\phi _1,\phi _2\in C^\infty (\rr{d}\times
	\rr{d_0};\rr d)$
	be $\SG$ mappings.
	\begin{enumerate}
		\item $\omega$ is called \emph{$(\phi,1)$-invariant}
		when
		\[
			\omega(\phi(x,\eta _1+\eta _2),\xi )\lesssim
			\omega
			(\phi(x,\eta _1),\xi ),
		\]
		for any $x, \xi \in \rr{d}$, $\eta _1,\eta _2\in \rr{d_0}$, uniformly with
		respect to $\eta _2\in \rr{d_0}$. The set of all $(\phi,1)$-invariant
		weights in $\mascP _{r,\rho}(\rr{2d})$ is denoted by
		$\mascP_{r,\rho}^{\phi,1}(\rr{2d})$;
		\item $\omega$ is called \emph{$(\phi,2)$-invariant}
		when
		\[
			\omega(x,\phi(\xi ,\eta _1+\eta _2))\lesssim
			\omega
			(x,\phi(\xi ,\eta _1)),
		\]
		for any $x, \xi \in \rr{d}$, $\eta _1,\eta _2\in \rr{d_0}$, uniformly with
		respect to $\eta _2\in \rr{d_0}$. The set of all $(\phi,2)$-invariant
		weights in $\mascP _{r,\rho}(\rr{2d})$ is denoted by
		$\mascP_{r,\rho}^{\phi,2}(\rr{2d})$;
		\item $\omega$ is called \emph{$(\phi _1,\phi _2)$-invariant}
		if $\omega$ is both $(\phi _1,1)$-invariant and $(\phi _2,2)$-invariant.
		The set of all $(\phi _1,\phi _2)$-invariant
		weights in $\mascP _{r,\rho}(\rr{2d})$ is denoted by
		$\mascP_{r,\rho}^{(\phi_1,\phi _2)}(\rr{2d})$
	\end{enumerate}
\end{defn}
%
%
%
%
%
%
%
%
%

\par

We now show that, under mild additional conditions, the
families of weights introduced in Subsection 
\ref{subs:1.1} are indeed ``invariant'' under composition
with $\SG$ maps with $\SG^0$-parameter dependence.
That is, the compositions introduced in Definition
\ref{def:omegainv} are still weight functions in the sense
of Subsection \ref{subs:1.1}, belonging to suitable sets
$\mathscr{P}_{r,\rho}(\rr{2d})$.

\par

\begin{lemma}\label{lemma:omegainv}
	Let $r,\rho\in[0,1]$, $r+\rho>0$, $\omega\in\mascP
	_{r,\rho}(\rr{2d})$, and let 
	$\phi\colon\rr{d}\times\rr{d}\to\rr{d}$
	be an $\SG$ map as in Definition \ref{def:sgmap}.
	The following statements hold true.
	\begin{enumerate}
		\item Assume $\omega\in\mascP_{1,\rho}^{\phi,1}(\rr{2d})$, 
		and set $\omega_1(x,\xi):=\omega(\phi(x,\xi),\xi)$.
		Then, $\omega_1\in\mascP_{1,\rho}(\rr{2d})$.
		\item Assume $\omega\in\mascP_{r,1}^{\phi,2}(\rr{2d})$, 
		and set $\omega_2(x,\xi):=\omega(x,\phi(\xi,x))$.
		Then, $\omega_2\in\mascP_{r,1}(\rr{2d})$.
	\end{enumerate}
\end{lemma}
\begin{proof}
	We prove only the first statement, since the proof of the
	second one follows by a completely similar argument,
	exchanging the role of $x$ and $\xi$.

\par
	
	It is obvious that $\omega_1\in C^\infty(\rr{d}\times
	\rr{d})$. The estimates \eqref{SGCondWeight} follows
	by F\`aa di Bruno's formula
	(cf. \cite{coriasco}). Explicitly, for $|\alpha
	+\beta|>0$,
	\begin{equation*}
		\norm{x}^{|\alpha |} \norm{\xi}^{\rho |\beta |} \partial
		^\alpha_x\partial^\beta_\xi\omega_1(x,\xi)
		=\norm{x}^{|\alpha|}\norm{\xi}^{\rho|\beta|} \partial
		^\alpha_x\partial^\beta_\xi(\omega(\phi(x,\xi),\xi))
	\end{equation*}
	belongs to the span of
	\begin{equation*}
		\left\{\norm{x}^{|\alpha|}\norm{\xi}
		^{\rho|\beta|}
		(\partial ^{\gamma _0}_x\partial ^{\delta_0}_\xi
		\omega )(\phi(x,\xi),\xi)
		\cdot\prod_{1\le j\le |\gamma _0|}\partial
		^{\gamma _j}_x\partial^{\delta _j}_\xi \phi (x,\xi)\;
		\colon \; \right.
		\left.\sum_{j\ge1}\gamma _j=\alpha ,\sum _{j\ge0} \delta
		_j=\beta \right \}.
	\end{equation*} 
	Denoting by $f_{\alpha \beta \gamma \delta}$,
	$\gamma =(\gamma _0,\gamma _1,\dots ,\gamma
	_{|\gamma_0|})$, 
	$\delta=(\delta_1,\dots,\delta_{|\gamma_0|})$, the terms
	in braces above, in view of the hypotheses we have
	\begin{multline*}
		|f_{\alpha \beta \gamma \delta}(x,\xi)|
		\\[1ex]
		\lesssim \norm{x}^{|\alpha |}
		\norm{\xi}^{\rho |\beta| }
		\cdot
		\omega (\phi (x,\xi ),\xi )\norm{\phi(x,\xi)}
		^{-|\gamma _0|}\norm{\xi}^{-\rho|\delta_0|}
		\cdot \prod _{1\le j\le |\gamma _0|}\norm{x}
		^{1-|\gamma _j|}\norm{\xi}^{-|\delta _j|}
		\\[1ex]
		\lesssim \omega(\phi(x,\xi),\xi)
		\cdot
		\norm{x}^{|\alpha|}\norm{\xi}^{\rho|\beta|}\cdot
		\norm{x}^{-|\gamma_0|}\norm{\xi}^{-\rho|\delta_0|}
		\cdot
		\norm{x}^{|\gamma_0|}
		\norm{x}^{-\sum_{j\ge1}|\gamma_j|}
		\norm{\xi}^{-\sum_{j\ge1}|\delta_j|}
		\\[1ex]
		= \omega_1(x,\xi)
		\cdot\norm{\xi}^{\rho|\beta|}\cdot\norm{\xi}^{-\rho|\beta|}
		=\omega_1(x,\xi),
	\end{multline*}
	%
%
%
%
%
%
	which implies \eqref{SGCondWeight} with $r=1$, $\rho\in[0,1]$, $|\alpha+\beta|>0$.
	The estimate for $\alpha=\beta=0$ is trivial. Then, \eqref{SGCondWeight} holds true for 
	$\omega_1$ with $r=1$, $\rho\in[0,1]$, as claimed. It remains to prove
	\eqref{moderate}. To this aim, observe that, by the moderateness
	of $\omega$, using the properties of $\phi$ we find, for some polynomial $v$,
	\begin{multline*}
		\omega_1(x+y,\xi+\eta)=\omega(\phi(x+y,\xi+\eta),\xi+\eta)
		\\
		=\omega\left(\phi(x,\xi+\eta)+\overbrace{\int_0^1
		\phi^\prime_x(x+ty,\xi+\eta)\cdot y\,dt}^{=z},
		\xi+\eta\right)
		\\
		\lesssim\omega(\phi(x,\xi+\eta),\xi)\,v(z,\eta).
	\end{multline*}
	%
%
%
%
	Since $|\phi^\prime_x(x+ty,\xi+\eta)|\lesssim1$ for any $x,y,\xi,\eta\in\rr{d}$, $t\in[0,1]$, so that
	$|z|\lesssim|y|$, we conclude, in view of the $(\phi,1)$-invariance of $\omega$, that
	\begin{multline*}
		\omega_1(x+y,\xi+\eta)\lesssim\omega(\phi(x,\xi+\eta),\xi)
		\cdot\tilde{v}(y,\eta)
		\\
		\lesssim\omega(\phi(x,\xi),\xi)\,\tilde{v}(y,\eta)
		\lesssim\omega_1(x,\xi)\,\tilde{v}(y,\eta),
	\end{multline*}
	%
%
%
%
	for some other suitable polynomial $\tilde{v}$ and any
	$x,y,\xi,\eta\in\rr{d}$. The proof is complete.
\end{proof}

\begin{rem}
\label{rem:3.2}
It is obvious that, when dealing with Fourier integral operators,
the requirements for $\phi$ and $\omega$ in Lemma \ref{lemma:omegainv}
need to be  satisfied only on the support of the involved amplitude.
By Lemma \ref{lemma:omegainv}, 
it also follows that if $a \in \SG^{(\omega)}_{1,1}(\rr{2d})$ and
$\phi=(\phi_1,\phi_2)$, where
$\phi_1\in  \SG^{1,0}_{1,1}(\rr{2d})$ and $\phi_2\in  \SG^{0,1}_{1,1}(\rr{2d})$
are $\SG$ maps with $\SG^0$ parameter dependence, 
then $a\circ \phi \in \SG^{({\omega _0})}_{1,1}(\rr{2d})$ when
${\omega _0}:=\omega \circ \phi$,
provided $\omega$ is $(\phi _1,\phi _2)$-invariant. Similar results
hold for $\SG$ amplitudes and weights defined on $\rr{3d}$.
\end{rem}

\par

\begin{rem}\label{rem:trivweight}
	By the definitions it follows that any weight
	$\omega=\vartheta_{s,\sigma}$, $s,\sigma\in\mathbf R$,
	 is $(\phi,1)$-,
	$(\phi,2)$-, and $(\phi_1,\phi_2)$-invariant with respect to any
	$\SG$ diffeomorphism 
	with $\SG^0$ parameter dependence $\phi$, $(\phi_1,\phi_2)$. 
%
%
%
%
%
%
\end{rem}

%
%

We conclude the section by recalling the definition, taken from \cite{coriasco},
of the sets of $\SG$ compatible cutoff and $0$-excision functions, which we 
will use in the sequel. By a standard construction, it is easy to prove that the 
sets $\Xi^\Delta(k)$ and $\Xi(R)$ introduced in Definition \ref{def:1.2.2} below
are non-empty, for any $k,R>0$.

\begin{defn}
\label{def:1.2.2}
The sets $\Xi^\Delta(k)$, $k > 0$, of the $\SG$ compatible cut-off functions
along the diagonal of $\rr{d}\times\rr{d}$, consist of all 
$\chi = \chi(x,y) \in  \SG ^{0,0}_{1,1}(\rr {2d})$ such that
\begin{equation}
\label{eq:1.1.3}
\begin{array}{rcl}
        |y-x| \le k\norm{x}/2 & \Longrightarrow & \chi(x,y) = 1,
        \\[1ex]
        |y-x| >         k      \norm{x} & \Longrightarrow & \chi(x,y) = 0.
\end{array}
\end{equation}
If not otherwise stated, we always assume $k \in (0,1)$.

\noindent
$\Xi(R)$ with $R > 0$ will instead denote the sets of all 
$\SG$ compatible $0$-excision functions, namely, 
the set of all $\varsigma = \varsigma(x, \xi) \in \SG^{0,0}_{1,1}(\rr {2d})$ such that
\begin{equation}
\label{eq:1.1.4}
\begin{array}{rcl}
        |x|+|\xi| \ge R           & \Longrightarrow & \varsigma(x, \xi) = 1,
        \\[1ex]
        |x|+|\xi| \le {R}/2 & \Longrightarrow & \varsigma(x, \xi) = 0.
\end{array}
\end{equation}
\end{defn}

\par

\section{Symbolic calculus for generalized FIOs of $\SG$ type}
\label{sec2}
%
%
%
%
We here introduce the class of Fourier integral operators we are interested in, 
generalizing those studied in \cite{coriasco}. In particular, we show how a symbolic
calculus can be developed for them. We examine their compositions with the generalized 
$\SG$ pseudo-differential operators introduced in \cite{CJT2}, and the compositions
between Type I and Type II operators. A key tool in the proofs of the composition
results below are the results on asymptotic expansions in the Weyl-H\"ormander
calculus obtained in \cite{CoTo}.

\par

%
%
%

\subsection{Phase functions of $\SG$ type}\label{subs:2.1}
We recall the definition of the class of admissible phase functions in 
the $\SG$ context, as it was given in \cite{coriasco}. We then observe
that the subclass of \emph{regular phase functions} generates
(parameter-dependent) mappings of $\rr d$ onto itself, which turn
out to be $\SG$ maps with $\SG^0$ parameter-dependence.
Finally, we define some \emph{regularizing operators},
which are used to prove the properties of the $\SG$ Fourier integral
operators introduced in the next subsection.

\par

\begin{defn}\label{def:2.1}
A real-valued function $\varphi \in \SG^{1,1}_{1,1}(\rr {2d})$
is called a \emph{simple phase function} (or \emph{simple phase}), if
\begin{equation}
\label{eq:2.0}
\norm{\varphi_{\xi}^\prime(x,\xi)} \asymp \norm{x}
       \mbox{ and }   
 \norm{\varphi_{x}^\prime(x,\xi)} \asymp \norm{\xi},
\end{equation} 
are fulfilled, uniformly with respect to $\xi$ and $x$, repectively. 
The set of all simple phase functions is denoted by $\Ph$.
Moreover, the simple phase function $\varphi$ is called \emph{regular},
if $\left|\det (\varphi^{\prime\prime}_{x \xi} (x,\xi) ) \right| \ge c$ for some $c>0$
and all $x,\xi \in \rr d$. The set of all regular phases is denoted by $\Phr$.
\end{defn}

\par

We observe that a regular phase function $\varphi$
defines two globally invertible mappings,
namely $\xi \mapsto  \varphi^\prime_x(x,\xi)$ and $x \mapsto  \varphi
^\prime_\xi (x,\xi)$, see the analysis in \cite{coriasco}. 
Then, the following result holds true for the mappings $\phi_1$ and $\phi_2$
generated by the first derivatives of the admissible regular phase functions.

\par

\begin{prop}
\label{prop:3.2}
Let $\varphi \in \Ph$. Then, for any $x_0,\xi _0\in \rr d$,
$\phi_1\colon\rr{d}\to\rr{d}\colon x \mapsto  \varphi^\prime_\xi (x,\xi _0)$ and
$\phi_2\colon\rr{d}\to\rr{d}\colon \xi \mapsto  \varphi^\prime_x(x_0,\xi)$ 
are $\SG$ maps (with $\SG^0$ parameter dependence), from $\rr{d}$ to itself. 
If $\varphi\in\Phr$, $\phi_1$ and $\phi_2$ give rise to $\SG$ diffeomorphism 
with $\SG^{0}$ parameter dependence.
\end{prop}

\par

For any $\varphi \in \Ph$, the operators $\Theta _{1,\fy}$ and $\Theta _{2,\fy}$
are defined by
$$
(\Theta _{1,\fy}f)(x,\xi ) \equiv f(\fy '_\xi (x,\xi ), \xi)
\quad \text{and}\quad
(\Theta _{2,\fy}f)(x,\xi ) \equiv f(x, \fy '_x(x,\xi )),
$$
when $f\in C^1(\rr {2d})$, and remark that
the modified weights
\begin{equation}\label{omegaVarphiDef}
(\Theta _{1,\fy}\omega )  (x,\xi ) = \omega (\fy '_\xi (x,\xi ), \xi)
\quad \text{and}\quad
(\Theta _{2,\fy}\omega ) (x,\xi ) = \omega (x, \fy '_x(x,\xi )), 
\end{equation}
%
%
%
%
will appear frequently in the sequel. In the following lemma we
show that these weights belong to the same classes of weights
as $\omega$, provided 
they additionally fulfill
\begin{equation}\label{WeightPhaseCond}
\Theta _{1,\fy}\omega  \asymp \Theta _{2,\fy}\omega 
\end{equation}
when $\fy$ is the involved phase function. That is,
\eqref{WeightPhaseCond} is a sufficient condition to obtain
$(\phi_1,1)$- and/or $(\phi_2,2)$-invariance of $\omega$
in the sense of Definition \ref{def:omegainv}, depending on the values of
the parameters $r,\rho\ge0$.

\par

\begin{lemma}
Let $\fy$ be a simple phase on $\rr {2d}$, $r,\rho \in [0,1]$ be such that
$r=1$ or $\rho =1$, and let $\Theta _{j,\fy}\omega$, $j=1,2$, be as in
\eqref{omegaVarphiDef}, where $\omega \in \mascP
_{r,\rho}(\rr {2d})$ satisfies \eqref{WeightPhaseCond}.
Then
\begin{equation*}
	\Theta _{j,\fy}\omega \in \mascP _{r,\rho}(\rr {2d}),\quad j=1,2.
\end{equation*}
\end{lemma}

\par

\begin{proof}
Evidently, the estimates
\eqref{SGCondWeight} for $\Theta _{1,\fy}\omega$
and $\Theta _{2,\fy}\omega$ follow from
Lemma \ref{lemma:omegainv}. We need to show
that $\Theta _{1,\fy}\omega$ and $\Theta _{2,\fy}\omega$ are moderate.

\par

By Taylor expansion, and the fact that $\omega$ is moderate,
there are numbers $\theta =\theta (x,y)\in [0,1]$ and $N_1\ge 0$
such that
\begin{multline*}
(\Theta _{1,\fy}\omega) (x+y,\xi ) = \omega (\fy _\xi '(x+y,\xi ),\xi )
= \omega (\fy _\xi '(x,\xi ) + \scal {\fy _{x,\xi } ''(x+\theta y,\xi )}y ,\xi )
\\[1ex]
\lesssim
\omega (\fy _\xi '(x,\xi ),\xi )
\eabs {\scal {\fy _{x,\xi } ''(x+\theta y,\xi )}y}^{N_1}
\lesssim \omega (\fy _\xi '(x,\xi ),\xi )\eabs y^{N_1}.
\end{multline*}
%
%
This gives
\begin{equation*}
(\Theta _{1,\fy}\omega )(x+y,\xi )\lesssim (\Theta _{1,\fy}\omega )(x,\xi )\eabs y^{N_1}.
\end{equation*}
In the same way we get
\begin{equation*}
(\Theta _{2,\fy}\omega )(x,\xi +\eta )\lesssim (\Theta _{2,\fy}\omega )(x,\xi )\eabs \eta ^{N_2},
\end{equation*}
for some $N_2\ge 0$. From these estimates we obtain
\begin{multline*}
(\Theta _{2,\fy}\omega )(x+y,\xi +\eta ) \lesssim
(\Theta _{2,\fy}\omega )(x+y,\xi )\eabs \eta ^{N_2}
\\
\asymp
(\Theta _{1,\fy}\omega )(x+y,\xi )\eabs \eta ^{N_2}
\lesssim
(\Theta _{1,\fy}\omega )(x,\xi )\eabs y^{N_1}\eabs \eta ^{N_2}
\\
\asymp 
(\Theta _{2,\fy}\omega )(x,\xi )\eabs y^{N_1}\eabs \eta ^{N_2}.
\end{multline*}
Hence $\Theta _{2,\fy}\omega$, and thereby $\Theta _{1,\fy}\omega$, are
$v$-moderate, when $v(x,\xi )=\eabs x^{N_1}\eabs \xi ^{N_2}$.
\end{proof}
%
%
%
%
%

\par

In the following lemma we establish mapping properties for the operators
$R_1$ and $\Div$, which, for $\varphi\in\Ph$, are defined by the formulas
\begin{equation}
\label{eq:2.2}
R_1 = \frac{1 - \Delta_{\xi}}{\norm{\varphi^\prime_{\xi} (x,\xi)}^2 - 
                            i \Delta_{\xi}\varphi(x,\xi)},
\end{equation}
and
\begin{equation}
\label{eq:2.2A}
(\Div a)(x,\xi ) =
\frac{a(x,\xi )}{\norm{\varphi^\prime_{\xi}(x,\xi ) }^2 - i \Delta_{\xi}\varphi (x,\xi )}.
\end{equation}
%
%
%
%
Here and in what follows we let
$$
^{t}a(x,\xi)=a(\xi,x)\quad  \text{and} \quad (a^*)(x,\xi)
=\overline{a(\xi,x)},
$$
when $a(x,\xi)$ is a function.

\par

\begin{lemma}
\label{lemma:2.2}
Let $\varphi\in\Ph$ and let $R_1$ and $\Div$ be defined by \eqref{eq:2.2} and
\eqref{eq:2.2A}. Then the following is true:
\begin{enumerate}
	\item $R_1e^{i\varphi} = e^{i\varphi}$;

\vrum

	\item $R_1 = \Div (1 - \Delta_{\xi})$:

\vrum

	\item for any positive integer $l$,
	\begin{equation}
		\label{eq:2.2.1}
		( {^t R_1} )^l = \underbrace{(1 - \Delta_{\xi}) \Div \cdots 
                           (1 - \Delta_{\xi}) \Div}_{\mbox{$l$ times}} = 
		\Div^l + Q_l(\Div, \Delta_{\xi}),
	\end{equation}
	where $Q_l(\Div, \Delta_{\xi})$ is a suitable differential operator
	depending on $l, \Div, \Delta_{\xi}$, whose terms 
	contains exactly $l$ factors equal to $\Div$ and at least one equal to $\Delta_{\xi}$.
	
\vrum

	\item If $\omega \in \mascP _{r,\rho}(\rr {2d})$, where $r,\rho \in [0,1]$ are such
	that $r+\rho >0$, then the mappings
	\begin{align*}
		\Div^l \;\colon &\SG^{(\omega)}_{r,\rho}(\rr {2d}) 
		\to 
		\SG^{(\omega\cdot\vartheta_{-2l,0})}_{r,\rho}(\rr {2d}),
	\\[1ex]
 		Q_l (\Div, \Delta_{\xi}) \; \colon &\SG^{(\omega)}_{r,\rho}(\rr {2d})
		\to 
		\SG^{(\omega\cdot\vartheta_{-2l,-2})}_{r,\rho}(\rr {2d})
	\end{align*}
	are continuous.
	\end{enumerate}
\end{lemma}

\par

The next lemma follows by straight-forward computations, using induction. The details are
left for the reader.
%
%
%

\par

\begin{lemma}
\label{lemma:2.1}
Let $\varphi \in \SG^{1,1}_{1,1}(\rr {2d})$, and let $\alpha$ and $\beta$
be multi-indices. Then $\partial _x^{\alpha}  \partial_{\xi}^{\beta}
e^{i \varphi(x,\xi)} =
         b_{\alpha ,\beta} (x,\xi) e^{i \varphi(x,\xi)}$, for some
%
%
         $b_{\alpha ,\beta} \in \SG^{|\beta |,|\alpha |}_{1,1}(\rr {2d})$.
\end{lemma}

\par

\subsection{Generalised Fourier integral operators of $\SG$
type}\label{subs:2.2}
In analogy with the definition of generalized $\SG$
pseudo-differential operators, recalled in Subsection
\ref{subs:1.1}, we define the class of Fourier integral
operators we are interested in terms of their distributional
kernels. These belong to a class of tempered
oscillatory integrals, studied in \cite{CoSch}. Thereafter we
prove that they posses convenient mapping properties.

\par

\begin{defn}\label{def:sgfios}
Let $\omega \in \mascP_{r,\rho}(\rr{2d})$ satisfy
\eqref{WeightPhaseCond}, $r,\rho\ge0$, $r+\rho>0$,
$\varphi\in\Ph$, $a,b\in\SG^{(\omega)}_{r,\rho}(\rr{2d})$.
\begin{enumerate}
\item The generalized Fourier integral operator $A=\op _\fy (a)$ of
\emph{$\SG$ type I} (\emph{$\SG$ FIOs of type I}) with phase $\varphi$
and amplitude $a$ is the linear continuous operator from $\cS(\rr d)$ to
$\cS^\prime(\rr d)$ with distribution kernel $K_A\in
\cS^\prime(\rr {2d})$ given by
\[
	K_A(x,y)=(2\pi)^{-d/2}(\cF_2(e^{i\fy}a))(x,y)\text ;
\]

\vrum

\item The generalized Fourier integral operator $B=\op _\fy ^*(b)$ of
\emph{$\SG$ type II} (\emph{$\SG$ FIOs of type II}) with phase
$\varphi$ and amplitude $b$ is the linear continuous operator
from $\cS(\rr d)$ to $\cS^\prime(\rr d)$ with distribution kernel
$K_B\in \cS^ \prime(\rr {2d})$ given by
\[
	K_B(x,y)=(2\pi)^{-d/2}(\cF^{-1}_2(e^{-i\fy }\overline b))(y,x).
\]
\end{enumerate}
\end{defn}

\par

Evidently, if $u \in \cS(\rr d)$, and $A$ and $B$ are the operators in Definition
\ref{def:sgfios}, then
\begin{align}
Au(x) &= \op_\varphi(a)u(x) = (2\pi)^{-d/2}\int e^{i \varphi(x, \xi)}
\, a(x, \xi) \, ({\cF{u}})(\xi)\,d \xi,\label{eq:0.1}
%
%
%
%
\intertext{and}
Bu(x) &= \op^*_\varphi(b)u(x)\notag
\\[1ex]
 &=(2\pi)^{-d} \iint e^{i(\langle x, \xi)- \varphi(y, \xi))}
                            \, \overline{b(y, \xi)} \, u(y) \,dy d\xi.\label{eq:0.1.0}
\end{align}
%
%
%
%
%
%
%
%
%

\par

\begin{rem}\label{rem:sgsymm}
In the sequel the formal ($L^2$-)adjoint of an operator $Q$ is denoted by $Q^*$.
By straightforward computations it follows that the $\SG$ type I
and $\SG$ type II operators are formal adjoints to each others, provided the
amplitudes and phase functions are the same. That is, if $b$ and $\varphi$
are the same as in Definition \ref{def:sgfios}, then
$\op^*_\varphi(b)=\op_\varphi(b)^*$.

\par

%
%
%
Obviously, for any $\omega \in \mascP_{r,\rho}(\rr {2d})$,
$^{t}\omega=\omega^*$ is also an admissible weight which belongs to
$\mascP_{\rho,r}(\rr {2d})$. Similarly, for arbitrary
$\varphi \in \Ph$ and $a \in \SG^{(\omega)}_{r,\rho}(\rr {2d})$, we have
$^{t}\varphi=\varphi ^*\in\Ph$ and $^{t}a, a^*\in\SG ^{(\omega ^*)} _{\rho,r}(\rr {2d})$.
Furthermore, by Definition \ref{def:sgfios} we get
\begin{equation}\label{eq:typeI-II}
\begin{gathered}
\op_\varphi^*(b) =  \cF^{-1} \circ \op_{- \varphi ^*}(b^*) \circ \cF^{-1}
\\[1ex]
\Longleftrightarrow
\\[1ex]
\op_{\varphi}(a) =  \cF \circ \op _{- \varphi ^*}^*(a^*) \circ \cF.
\end{gathered}
\end{equation}
%
%
%
%
%
\end{rem}

\par

The following result shows that type I and type II operators
are linear and continuous from $\cS (\rr d)$ to itself, and
extendable to linear and continuous operators from
$\cS^\prime(\rr d)$ to itself.

\par

\begin{thm}
\label{thm:2.1}
Let $a$, $b$ and $\varphi$ be the same as in Definition \ref{def:sgfios}.
Then $\op _\fy (a)$ and $\op _\fy ^*(b)$ are linear and continuous
operators on $\cS (\rr d)$, and uniquely extendable to linear and
continuous operators on $\cS '(\rr d)$.
\end{thm}

\par

\begin{proof}
First we consider the operator $\op _\fy (a)$. By differentiation under
the integral sign,  using Lemma \ref{lemma:2.1} and the facts
that differentiations and multiplications by polynomials maps $\SG$
classes into $\SG$ classes, it is enough to prove that
$$
|Au(x)|\lesssim p(u),\quad u\in \cS (\rr d),
$$
for some seminorm $p$ on $\cS(\rr d)$. By a regularization argument,
using the operator $R_1$ defined in \eqref{eq:2.2.1},
in view of Lemma \ref{lemma:2.2} we find, for arbitrary $l$ and 
$\mathfrak{D}=\norm{\varphi^\prime_{\xi} }^2 - i \Delta_{\xi}\varphi$,
\begin{align*} 
&
Au(x) = 
     (2\pi)^{-d}\!\!\!\int e^{i \varphi(x, \xi)} ({^t}R_1)^l[ a(x, \xi)
               (\cF u)(\xi)] d\xi =
\\
&
= (2\pi)^{-d}\!\!\!\int e^{i \varphi(x, \xi)} \!\!\left\{
               \frac{a(x, \xi)}{(\mathfrak{D}(x,\xi))^l}{(\cF u)}(\xi) +             
               Q_l(\cD , \Delta_{\xi}) 
                 \left[
                    a(x, \xi) {(\cF u)}(\xi)
                 \right]
               \right \}
               d \xi
\\
\label{eq:2.5}
& = (2\pi)^{-d}\!\!\!\int e^{i \varphi(x, \xi)} \!\!\left [
               \frac{a(x, \xi)}{(\mathfrak{D}(x,\xi))^l}{(\cF u)}(\xi) +             
               \sum_{|\gamma| \le 2l} c_{\gamma}(x,\xi)
               D^\gamma {(\cF u)}(\xi)
               \right ]
               d \xi                                    
\end{align*}
with coefficients $c_{\gamma} \in \SG^{(\omega \cdot \vartheta
_{-2l,-2})}_{r,\rho}(\rr {2d})$ depending only 
on $a$ and $\mathfrak{D}$, and $\dfrac{a(x, \xi)}{(\mathfrak{D}(x,\xi))^l}\in
\SG^{(\omega\cdot\vartheta_{-2l,0})}_{r,\rho}(\rr {2d})$. Since $\omega$
is polynomially bounded and
$u\in \cS (\rr d)$, it follows that, for any $l$ and a suitable $m\in \mathbf R$,
there is a semi-norm $p$ on $\cS$ such that
\[
	|Au(x)|\lesssim \norm{x}^{m-2l} \, p(u) \int \norm{\xi}^{-d-1} d \xi \lesssim p(u),
\]
as desired, choosing $l$ and $k$ large enough.
The $\cS$-continuity of the operators of type II follows by similar argument.
The details are left for the reader.

\par

Finally, the continuity and uniqueness on $\cS ^\prime(\rr d)$ of the operators
$\op _\fy (a)$ and $\op _\fy ^*(b)$
now follows by duality, recalling Remark \ref{rem:sgsymm}.
\end{proof}

\subsection{Compositions with pseudo-differential operators of $\SG\!$ type.}
\label{subs:2.3}
The composition theorems presented in this and the subsequent subsections
are variants of those originally appeared in \cite{coriasco}. We include
anyway some of their proofs, focusing on the role of the parameters in the 
classes of the involved amplitudes and symbols, as well as on 
the different notion of asymptotic expansions needed here, see \cite{CoTo}.
The notation used in the statements of the composition theorems are those
introduced in Subsections \ref{subs:1.3} and \ref{subs:2.1}.

\par

\begin{thm}
\label{thm:0.1}
Let $r_j,\rho _j\in [0,1]$, $\varphi \in \Ph$ and let $\omega _j
\in \mascP_{r_j,\rho _j}(\rr {2d})$, $j=0,1,2$, be such
that
$$
\rho _2=1,
\quad r_0=\min\{r_1,r_2,1\} ,\quad \rho _0=\min\{ \rho_1,1\},
\quad \omega _0 =\omega_1\cdot (\Theta _{2,\fy}\omega _2),
$$
and $\omega _2\in \mathscr{P}_{r,1}(\rr{2d})$ is
$(\phi,2)$-invariant with respect to $\phi\colon
\xi\mapsto\varphi^\prime_x(x,\xi)$.
Also let $a \in \SG^{(\omega _1)} _{r_1,\rho_1}(\rr {2d})$,
$p \in \SG^{(\omega _2)}_{r_2,1}(\rr {2d})$, and let
\begin{equation}
\label{eq:0.4}
\psi(x,y,\xi) = \varphi(y,\xi) - \varphi(x,\xi) -
\scal{ y - x}{\varphi '_x(x,\xi)}.
\end{equation}
Then
\begin{alignat*}{2}
\op(p) \circ \op_\varphi(a) &= \op_{\varphi}(c) \operatorname{Mod}
\op _\varphi (\SG ^{(\omega \vartheta _{0,-\infty })}_0 ),& \quad r_1=0,
\\[1ex]
\op(p) \circ \op_\varphi(a) &= \op_{\varphi}(c) \operatorname{Mod}
\op (\mathscr S ),& \quad r_1>0,
\end{alignat*}
where $c \in \SG^{(\omega _0)}_{r_0,\rho _0}(\rr {2d})$
admits the asymptotic expansion
\begin{equation}
\label{eq:0.3}
c(x,\xi) \sim \sum_{\alpha} \frac{i^{|\alpha|}}{\alpha!} 
(D^\alpha_\xi p)(x, \varphi '_x(x,\xi))
\,D^\alpha_y \!\!\left[ e^{i \psi(x,y,\xi)} a(y,\xi) \right]_{y=x}.
\end{equation}
\end{thm}
%


As usual, we split the proof of Theorem \ref{thm:0.1} into various intermediate steps. We first need
an expression for the derivatives of the exponential functions appearing in \eqref{eq:0.3}. Again,
Lemma \ref{lemma:3.1} is a special case of the Fàa di Bruno
formula, and can be proved by induction. For the proof of Lemma \ref{lemma:3.2}, see \cite{coriasco}.
Then, in view of these two results, in Lemma \ref{lemma:3.3} we can prove that the terms which appear
in the right-hand side of \eqref{eq:0.3} indeed give a generalized $\SG$ asymptotic expansion,
in the sense described in Definition \ref{def:gensgasexp}. and \cite{CoTo}.

\par

\begin{lemma}
\label{lemma:3.1}
Let $\fy \in C^\infty (\rr {2d})$, and let $\psi$ be as in \eqref{eq:0.4}.
If $\alpha \in \nn d$ satisfies $|\alpha |\ge 1$, then
\begin{gather}
D^{\alpha}_y e^{i \psi} =
\tau_{\alpha} e^{i \psi}\notag
\intertext{where}
\tau _\alpha 
=
\left ( \varphi '_y - \varphi '_x \right ) ^\alpha
+ 
\sum _{j} c_{j}
\left ( \varphi '_y - \varphi '_x\right)^{\delta _{j}}
      \prod_{k=1}^{N_j}
      D ^{\beta_{jk}}_y \varphi
       \label{eq:3.1}
\end{gather}
for suitable constants $c_{j}\in\mathbf R$, and the summation in last sum should be
taking over all multi-indices $\delta _j$ and $\beta _{jk}$ such that
\begin{equation}
\label{eq:3.3}
\delta _{j} + \sum_{k=1}^{N_{j}} \beta_{jk} = \alpha,
\quad \text{and}\quad
|\beta_{jk}| \ge 2.
\end{equation}

\par

%
%
In \eqref{eq:3.1},
$\varphi '_{x} = \varphi '_{x}(x,\xi)$, 
$\varphi '_{y} = \varphi '_{y}(y,\xi)$
and
$\partial^{\alpha}_y \varphi = \partial^{\alpha}_y \varphi (y,\xi)$
is to be understood.
%
%
%
%
\end{lemma}

\par

Note that, by \eqref{eq:3.3}, we have, in each term appearing in \eqref{eq:3.1},
\begin{equation}
\label{eq:3.4}
|\alpha| \ge
\sum_{k=1}^{N_{j}} |\beta_{jk}| \ge 2 N_j
\Rightarrow 
N_j \le \frac{|\alpha|}{2}.
\end{equation}
%
%
%
%
%
%

\par

\begin{lemma}
\label{lemma:3.2}
Let $\varphi \in \SG^{1,1}_{1,1}(\rr {2d})$, and let $\psi$ be as in \eqref{eq:0.4}.
If $\alpha \in \nn d$ satisfies $|\alpha | \ge 1$, then
\[
	\begin{aligned}
		\left. \partial^{\alpha}_y e^{i \psi(x,y,\xi)} \right|_{y=x}
		&\in \SG^{[-|\alpha|/2], [|\alpha|/2]}_{1,1}(\rr {2d})
		\\
		&\Rightarrow
		\left. \partial^{\alpha}_y e^{i \psi(x,y,\xi)} \right|_{y=x}
		\lesssim \vartheta_{-|\alpha|/2,|\alpha|/2}(x,\xi).
	\end{aligned}
\]
Moreover, $|y-x| \le \varepsilon_1\norm{x}, \varepsilon_1 \in (0,1)$,
implies that each summand in the right-hand side of \eqref{eq:3.1} can be estimated by the product
of a suitable power $|y-x|^{m_0}$ times a weight of the form $\norm{x}^m\norm{\xi}^\mu$, with
$0\le m_0\le\mu\le|\alpha|$, $m\le-\frac{|\alpha|}{2}$.
\end{lemma}

\par

\begin{lemma}
\label{lemma:3.3}
Let $\varphi \in \Ph$, $\psi$ be as in \eqref{eq:0.4}, and let
$a$, $p$, $\omega _j$, $r_j$ and $\rho _j$, $j=0,1,2$, be as
in Theorem \ref{thm:0.1}. Then
%
%
%
\begin{equation}
\label{eq:3.8}
\begin{gathered}
\sum _{\alpha} \frac{c_{\alpha}(x,\xi)}{\alpha!},
%
%
%
\\[1ex]
\text{with}\quad 
c_{\alpha}(x,\xi) = i^{|\alpha|}
(D^\alpha_\xi p)(x, \varphi^\prime_x(x,\xi))
\,
D^\alpha_y \!\!\left[ e^{i \psi(x,y,\xi)} a(y,\xi) \right]_{y=x}
\end{gathered}
\end{equation}
is a generalized $\SG$ asymptotic expansion which defines an amplitude 
$c \in \SG^{(\omega _0)}_{r_0,\rho _0}(\rr {2d})$, modulo a remainder
of the type described in \eqref{eq:gensgasexp}. 
\end{lemma}

\par

\begin{proof}
Using Lemma \ref{lemma:3.2}, the hypothesis
$a \in \SG^{(\omega _1)}_{r_1,\rho _1}$, and
the properties of the symbolic calculus, we see that
\begin{eqnarray*}
	D^\alpha_y \left[ e^{i \psi(x,y,\xi)} a(y,\xi) \right ]_{y=x}
	&   =   & \left.
	\sum_{0 \le \beta \le \alpha} \binomial{\alpha}{\beta}
                                        D^{\beta}_y e^{i \psi(x,y,\xi)}
                                        D^{\alpha-\beta}_y a(y,\xi)
	\right|_{y=x}
	\\
	& \in & \sum_{0 \le \beta \le \alpha}
	\SG^{(\vartheta_{-|\beta|/2, |\beta|/2})}_{1, 1}
	\cdot
	\SG^{(\omega _0\cdot\vartheta_{-r_1( |\alpha| - |\beta|), 0})}_{r_1, \rho_1}
	\\
	& = & \sum_{0 \le \beta \le \alpha}
	\SG^{(\omega_1\cdot\vartheta_{-r_1 |\alpha| +(r_1-1/2) |\beta|), |\beta |/2})}
	_{\min\{r_1,1\}, \min\{\rho _1,1 \} }
	\\
	&\subseteq & \SG^{(\omega_1\cdot\vartheta_{-\min\{r_1,1/2\} |\alpha|, |\alpha |/2})}
	_{\min\{r_1,1\}, \min\{\rho _1,1 \} }.
\end{eqnarray*}
Using $\varphi \in \Ph$, in particular \eqref{eq:2.0}, and the results in Subsections
\ref{subs:1.3} and \ref{subs:2.1}, we also easily have:
\begin{eqnarray*}
	(D^\alpha_\xi p)(x, \varphi^\prime_x(x,\xi))
	\in
	\SG^{(\Theta _{2,\fy}\omega_{2}\cdot \vartheta_{0, -|\alpha|})}_{r_2, 1}.
\end{eqnarray*}
Summing up, we obtain, for any multi index $\alpha$,
\[
	c_{\alpha}(x,\xi) \in 
	\SG^{(\omega_2\cdot\vartheta_{-\min\{r_1,1/2\}
	|\alpha |, - |\alpha |/2})}_{\min\{r_1,r_2,1\}, \min
	\{\rho _1,1\}},
\]
which proves the lemma, by the hypotheses and the general
properties of the symbolic calculus.
\end{proof}

\par

The next two lemmas are well-known, see, e.g., \cite{Co, coriasco},
and can be proved by induction on $l$.

\par

\begin{lemma}
\label{lemma:3.4}
Let
$$
\Omega  = \sets {(x,y,\eta )\in \rr {3d}}{|x-y|>0},
$$
and let $R_2$ be the operator on $\Omega$, given by
\begin{equation}
	\label{eq:defr2}
	R_2 =  \sum_{j=1}^d \frac{x_j - y_j}{|x - y|^{2}} D_{\eta_j}.
\end{equation}
Then $R_2 e^{i \scal{x - y}{\eta}}$ $= e^{i \scal{x - y}{\eta}}$ when
$(x,y,\eta ) \in \Omega$, and for any positive integer $l$,
\[
(^t R_2)^l = \sum_{|\theta| = l} c_{\theta}
           \frac{(x-y)^\theta}{|x-y|^{2l}} D_{\eta}^\theta,
\]
for suitable coefficients $c_{\theta}$.
\end{lemma}

\par

\begin{lemma}
\label{lemma:3.4.1}
Let $\Omega \subseteq \rr d$ be open, $f\in C^\infty (\Omega)$ be such that
$|f^\prime_y(y)| \ne 0$, and let
\begin{equation}
\label{eq:3.10.1}
R_3 = \frac{1 }{|f^\prime_{y}(y)|^{2}}
    \sum_{k=1}^d  f^\prime_{y_k}(y) D_{y_k}.
\end{equation}
Then $R_3 e^{i f} = e^{i f}$, and for any positive integer $l$,
\begin{equation}
\label{eq:3.10.2}
                   ({^t R_3})^l = \frac{1}{|f^\prime_y(y)|^{4l}} 
                                \sum_{|\alpha| \le l}
                                   P_{l\alpha}(y) D_y^{\alpha} ,              
\end{equation}
with
\begin{equation}
\label{eq:3.10.3}
P_{l\alpha} = \sum c^{l\alpha}_{\gamma\delta_{1}\cdots
                        \delta_{l}}
                        (f^\prime_y)^\gamma\,
                        D_y^{\delta_{1}} f
                        \cdots  
                        D_y^{\delta_{l}} f,
\end{equation}
where the last sum should be taken over all $\gamma$
and $\delta$ such that
\begin{equation}
\label{eq:3.10.4}
|\gamma| = 2l \quad \text{and}\quad
|\delta_{j}| \ge 1, \;\; \sum_{j=1}^l |\delta_{j}| + |\alpha| = 2l,
\end{equation}
and $c^{l\alpha}_{\gamma\delta_{1}\cdots \delta_{l}}$ are suitable constants.
\end{lemma}
%
%
%
%

\par

\begin{lemma}
\label{lemma:3.5}
Let $\varphi, a, p, r_j, \rho _j$ be as in Theorem \ref{thm:0.1},
$\chi \in \Xi^\Delta(\varepsilon_1)$, and let
\[
h(x,\xi) = (2\pi)^{-d}\iint 
               e^{i (\varphi(y,\xi) - \varphi(x,\xi)
                        - \scal{y-x}{\eta} )}
               (1-\chi(x,y)) a(y,\xi) p(x,\eta)\, d y d \eta .
\]
Then $h\in \cS(\rr {2d})$.
\end{lemma}

\par

For the proof of Lemma \ref{lemma:3.5} we recall that for every $\ep >0$ it exists
an $\ep _0>0$ such that
\begin{equation}
\label{eq:3.10}
|y - x| \ge \ep _0  \norm{y} \quad \text{when}\quad
|y - x| \ge \varepsilon  \norm{x}.
\end{equation}
Hence,
\begin{equation}
\label{eq:3.10AA}
(\norm{x} \norm{y})^{\frac{1}{2}} \le \norm{x} + \norm{y} \lesssim |y - x|\quad
\text{when}\quad |y - x| \ge \varepsilon  \norm{x}.
\end{equation}

\par

\begin{proof}
We make use of the operators
$$
\widetilde{R}_1 = \frac{1 - \Delta_{y}}{\norm{\varphi^\prime_y(y,\xi)}^2 - i
\Delta_{y}\varphi(y,\xi)},
$$
which has properties similar to those of 
the operator $R_1$ defined in \eqref{eq:2.2},
and $R_2$, defined in \eqref{eq:defr2}. For any
couple of positive integers $l_1,l_2$ we have
\begin{eqnarray}
\nonumber
h(x,\xi) & \!\!\!\!\! = \!\!\!\!\!\! &
                   (2\pi)^{-d}\iint  \!\!
                   e^{i (\varphi(y,\xi) - \varphi(x,\xi)
                           - \scal{y-x}{\eta} )}
                   (1-\chi(x,y)) a(y,\xi)
                   \!\left[ ( {^t} R_2 )^{l_2} p \right] \!\!(x,\eta)\, d y d \eta
\\
\label{eq:3.12}
             & \!\!\!\!\! = \!\!\!\!\!\! &
                   (2\pi)^{-d}\iint \!\!
                   e^{i (\varphi(y,\xi) - \varphi(x,\xi)
                           + \scal{x}{\eta})}
                   ( {^t} \widetilde{R}_1)^{l_1} \!\left[
                                   e^{- i \scal{ y }{ \eta } }          
                                   q(x,y,\xi,\eta)
                             \right]
                              d y d \eta
\end{eqnarray}
when
\[
q(x,y,\xi,\eta) = (1-\chi(x,y)) a(y,\xi)\!
                  \left[ ( {^t} R_2 )^{l_2} p \right] \!\!(x, \eta).
\]
By Lemma \ref{lemma:3.4}, we 
get
\begin{align*}
\partial^\alpha_{y} q(x,y, \xi&,\eta) =
\\
 = & \partial^\alpha_{y} \left[
                                 (1 - \chi(x,y)) a(y,\xi)
                                 \sum_{|\theta|=l_2}
                                 c_{\theta}
                                 \frac{(x-y)^\theta}{|x-y|^{2l_2}}
                                 (D^\theta_{\eta}p)(x, \eta)
                          \right]
\\
= &\sum_{|\theta|=l_2} (D^\theta_{\eta}p)(x, \eta)
      \sum_{\alpha_{1}+\alpha_{2}+\alpha_{3} = \alpha}
            \frac{\alpha!}{\alpha_{1}!\alpha_{2}!\alpha_{3}!}
            (\delta_{|\alpha_{1}|,0} - (\partial_y^{\alpha_{1}} \chi)(x,y))\cdot
\\
&       \cdot  (\partial_y^{\alpha_{2}}a)(y,\xi)
            \sum_{\beta_{1}+\beta_{2} = \alpha_{3}}
            \frac{\alpha_{3}!}{\beta_{1}!\beta_{2}!} c_{\theta\beta_{1}}
            (x-y)^{\theta-\beta_{1}} 
            \frac{P_{\beta_{2}}(x-y)}{|x-y|^{2(r_2+|\beta_{2}|)}},    
\end{align*}
with $P_{\beta_{2}}$ homogeneous polynomial of degree $|\beta_{2}|$,
while $\delta_{|\alpha_{1}|,0}=1$
for $\alpha_1=0$, $\delta_{|\alpha_{1}|,0}=0$ otherwise. Then we obtain
\begin{multline*}
|\partial^\alpha_{y} q(x,y,\xi,\eta)|
\\
\lesssim
      \sum_{|\theta|=l_2} \omega _2(x,\eta) \vartheta_{0,-|\theta|}(x,\eta)
      \sum_{\alpha_{1}+\alpha_{2}+\alpha_{3} = \alpha}
            \norm{y}^{-|\alpha_{1}|}\omega _1(y,\xi)
            \vartheta_ {-r_1|\alpha _{2}|,0}
\\
	\sum_{\beta_{1}+\beta_{2} = \alpha_{3}}
            |x-y|^{|\theta|-|\beta_{1}|+|\beta_{2}|-2l_2-2|\beta_{2}|}
\\
   \lesssim
       \; \omega _1(x,\eta) \omega _2(y,\xi) \cdot \vartheta_{0,- l_2}(x,\eta)\cdot
\\
      \sum_{\alpha_{1}+\alpha_{2}+\alpha_{3} = \alpha}
            \vartheta_{-\min\{r_1,1\}(|\alpha _{1}|+|\alpha _{2}|),0}(y,\xi)
            |x-y|^{-l_2-|\alpha _{3}|}.
\end{multline*}
In view of the fact that $|y-x| \ge \frac{\varepsilon_1}{2} \norm{x}$ on $\supp(q)$, from
\eqref{eq:3.10} and \eqref{eq:3.10AA} we also obtain
\begin{eqnarray*}
|y-x| \ge \frac{\varepsilon_1}{2} \norm{x} 
\Rightarrow
     |y-x| \gtrsim \norm{y}
\Rightarrow
     |y-x| \gtrsim \norm{x} + \norm{y} \ge (\norm{x} \norm{y})^{\frac{1}{2}},
\end{eqnarray*}
and we can conclude
\begin{multline}
\label{eq:3.13}
|\partial^\alpha _{y} q(x,y,\xi,\eta)|
\\[1ex]
         \lesssim
         \omega _1(x,\eta) \omega _2(y,\xi)\cdot 
         \vartheta _{-l_2/2,-l_2/2}(x,y)\cdot \eabs \eta ^{-l_2}
         \eabs y^{-\min\{r_1,1/2\}|\alpha |}.
\end{multline}

\par

Finally, since admissible weight functions are
polynomially moderate, it follows by choosing $l_2$ large enough
that that the order of $q$ can be made arbitrary low with respect to
$x,y,\eta$. Moreover, when
derivatives with respect to $y$ are involved, $q$ behaves
as an $\SG$ symbol.

\par

We now estimate the integrand of \eqref{eq:3.12}. As shown
in Lemma \ref{lemma:2.2}, we have
\begin{eqnarray*}
\nonumber
& &\hspace*{-4mm}
( {^t} \widetilde{R}_1)^{l_1} \left[ e^{- i \scal{y}{\eta}}          
                 q(x,y,\xi,\eta) \right] =
\\
& &\hspace*{-4mm}
\label{eq:3.14}
= e^{- i \scal{y}{\eta }} \frac{q(x,y,\xi,\eta)}{(\norm{\varphi^\prime_\xi(y,\xi)}^2 - 
                            i \Delta_{\xi}\varphi(y,\xi))^{l_1}} +
     Q(\Div, \Delta_{y}) \left[ e^{- i \scal{y}{\eta}}          
                 q(x,y,\xi,\eta) \right]\!,
\end{eqnarray*}
as in \eqref{eq:2.2.1}. Due to the presence of the exponential in the 
argument of $Q(\Div, \Delta_{y})$, in the second term there are
powers of $\eta$ of height not greater than $2l_1$. Owing to 
\eqref{eq:3.13} we finally find
\begin{eqnarray*}
              |x^\alpha \xi^\beta h (x,\xi)| 
& \lesssim &
              \norm{x}^{-\frac{l_2}{2}+|\alpha|}
              \norm{\xi}^{-2l_1+|\beta|}
\\
&       &
              \int \omega(y,\xi)\,\norm{y}^{-\frac{l_2}{2}}dy
              \int \omega_0(x,\eta)\,\norm{\eta}^{-\frac{l_2}{2}+2l_1} d\eta  \lesssim
              1,
\end{eqnarray*}
for all multi-indices $\alpha,\beta$, provided that $l_1$ and $l_2$ are large enough, 
since $\omega _1$ and $\omega _2$ are polynomially bounded. Here $l_1$
is chosen first, and thereafter $l_2$ is fixed accordingly.
Differentiating $h_2$ and multiplying it by powers of $x$ and $\xi$
would give a linear combination of expressions similar to \eqref{eq:3.12}, with
different $\omega _1$, $\omega _2$ and parameters for the involved symbols, which
are then similarly estimated by constants. The proof is complete.
\end{proof}

\par

\begin{proof}[Proof of Theorem \ref{thm:0.1}]
Let
\[
c(x,\xi) = (2\pi)^{-d}\iint 
           e^{i (\varphi (y,\xi ) - \varphi (x,\xi) - \scal{y-x}{\eta})}
           a(y,\xi) p(x,\eta)\,dy d\eta
\]
By explicitly writing $\op(p)\circ\op_{\varphi}(a) u(x)$ with $u\in\cS$, we obtain
\begin{multline*}
\op(p)\circ\op_{\varphi}(a) u(x) =
\\
=
    (2\pi)^{-3d/2}\int  e^{i \scal{x}{\xi}} p(x, \xi)
    \int  e^{-i \scal{y}{\xi}}
    \int  e^{i \varphi(y,\eta)} a(y,\eta) \widehat u(\eta)\, d \eta d yd \xi
\\
=
(2\pi)^{-d/2}
\int  e^{i \varphi(x,\eta)}
     c(x,\eta )
     \widehat u(\eta)\, d \eta
\\
=
(2\pi)^{-d/2}
\int e^{i \varphi(x,\xi)}
     c(x, \xi)
     \widehat u(\xi)\, d \xi.
\end{multline*}
We have to show that $c\in \SG^{(\omega _0)}_{r_0,\rho _0}$.
Choosing $\chi \in \Xi^\Delta(\varepsilon_1)$, with $\varepsilon_1\in(0,1)$
fixed below (after equation \eqref{eq:3.37}), we write
$c=c_0+h$, where
\begin{align*}
c_0(x,\xi )& = 
    (2\pi)^{-d}\iint 
           e^{i (\varphi(y,\xi) - \varphi(x,\xi) - \scal{y-x}{\eta})}
           \chi(x,y) a(y,\xi) p(x,\eta)\,dy d\eta
\intertext{and}
h(x,\xi )& = 
   (2\pi)^{-d} \iint 
           e^{i (\varphi(y,\xi) - \varphi(x,\xi) - \scal{y-x}{\eta})}
           ( 1 - \chi(x,y) ) a(y,\xi) p(x,\eta)\,dy d \eta .
\end{align*}
By Lemma \ref{lemma:3.5} we get $h \in \cS$. We shall prove that
$c_0 \in \SG^{(\omega _0)}_{r_0,\rho _0}$, and admits the asymptotic 
expansion in Lemma \ref{lemma:3.3}. 

\par

In fact, let
$\eta = \varphi '_{x}(x,\xi) + \theta$. Then
\begin{eqnarray*}
p(x,\eta) & = &
          \sum_{|\alpha| < M}
            \frac{i^{|\alpha |}\theta ^\alpha}{\alpha!}
            (D^\alpha _{\xi} p)(x, \varphi '_{x}(x,\xi))
          +
          \sum_{|\alpha| = M}
            \frac{ i^{|\alpha |} \theta^\alpha}{\alpha!} 
             r_{\alpha} (x,\xi,\theta)
\\
r_{\alpha} (x,\xi,\theta) & = & 
          M\int_{0}^{1}  (1-t)^{M-1}
            (D^\alpha _{\xi} p)(x, \varphi '_{x}(x,\xi) + t\theta)\, dt,
\end{eqnarray*}
by Taylor's formula. Also let
\begin{multline*}
H_\alpha (x,\xi ,\theta ) =  \theta ^\alpha \cF  \left (
                       e^{i \psi(x,\cdo ,\xi)}
                     \chi(x,\cdo ) a(\cdo ,\xi) 
                   \right )(\theta )
                   \\[1ex]
                   = \cF  \left (
                     D^\alpha ( e^{i \psi(x,\cdo ,\xi)}
                     \chi(x,\cdo ) a(\cdo ,\xi) )
                   \right )(\theta ) .
\end{multline*}
Then
$$
c_0(x,\xi ) = c_{0,1}(x,x,\xi )+c_{0,2}(x,x,\xi ),
$$
where
\begin{align*}
             c_{0,1}(x,y,\xi )
             &= \sum_{|\alpha| < M}
             \frac{(i^{|\alpha |}D^\alpha _{\xi} p)(x,
                   \varphi '_x(x,\xi))}{\alpha!}
             (\cF ^{-1}H_\alpha (x,\xi ,\cdo ))(y)
             \\
             c_{0,2}(x,y,\xi )
             &= \sum_{|\alpha| =M}
             \frac{i^{|\alpha |}}{\alpha!}
             (\cF^{-1}(r_\alpha (x,\xi ,\cdo )H_\alpha (x,\xi ,\cdo ))(y)
\end{align*}

\par

Now, since every derivative of $\chi$ vanishes in a neighbourhood of 
the diagonal of $\rr d \times \rr d$, and 
$\chi(x,x) = 1$, we get
\begin{align*}
c_{0,1}(x,x,\xi ) &= \sum_{|\alpha| < M} \frac{c_{\alpha}(x,\xi )}{\alpha !} 
\\
c_{0,2}(x,x,\xi ) &= \sum_{|\alpha| = M} \frac{c_{0,\alpha}(x,\xi )}{\alpha!},
\end{align*}
where $c_{\alpha}$ is the same as in \eqref{eq:3.8}, and 
\begin{eqnarray*}
c_{0,\alpha}(x,\xi) = (2\pi)^{-d/2}\int e^{i \scal{x}{\theta}}
                 r_{\alpha} (x,\xi,\theta)
                 H_\alpha (x,\xi ,\theta )
                  d \theta .
\end{eqnarray*}

\par

By the properties of the generalized $\SG$ asymptotic expansions,
we only have to estimate $c_{0,\alpha}$, $|\alpha |=M$ to complete the proof (cf. \cite{CoTo}).

\par

Let
$\chi _{0,\xi} = \chi _0(\norm \xi ^{-1}\cdo )$, where $\chi _0 \in C^\infty_{0}(\rr d)$
is identically equal to $1$ in the ball $B_{\varepsilon_2/{2}}(0)$ and supported in the ball
$B_{\varepsilon_2}(0)$, where $\varepsilon_2\in(0,1)$ will be fixed later
(after equation\eqref{eq:3.20}). Then,
$$
\supp {\chi _{0,\xi } }\subset B_{\varepsilon_2\norm{\xi}}(0).
$$

\par

Next we split $c_{0,\alpha}$ into the sum of the two integrals
\begin{align*}
c_{1,\alpha}(x,\xi )  & =  \!\!
        (2\pi)^{-d/2}\int e^{i \scal{x}{\theta}}
                 r_{\alpha} (x,\xi,\theta)
                 \chi _{0,\xi }(\theta )
                H_\alpha (x,\xi ,\theta )
                d \theta ;
\\
c_{2,\alpha}(x,\xi )  & =  \!\!
        (2\pi)^{-d/2}\int e^{i \scal{x}{\theta}}
                 r_{\alpha} (x,\xi,\theta)
                 \big ( 1-\chi _{0,\xi }(\theta )\big )
                H_\alpha (x,\xi ,\theta )
                d \theta ;
\end{align*}
We claim that for some integer $N_0\ge 0$, depending on $\omega _2$ only,
it holds
\begin{align}
|c_{1,\alpha}(x,\xi )|  & \lesssim \omega _1(x,\xi )(\Theta _{2,\fy}\omega _2)(x,\xi )
\norm x^{-\min (r_1,1/2)|\alpha |}\norm \xi ^{N_0-|\alpha |/2},\label{c1alphaEst}
\intertext{and that for every integers $N_1$ and $N_2$ it holds}
|c_{2,\alpha}(x,\xi )|  & \lesssim \norm x^{-N_1}\norm \xi ^{-N_2}.\label{c2alphaEst}
\end{align}

\par

In order to prove \eqref{c1alphaEst} we set
      \[
      f_{\alpha} (x,\xi,y) = \cF^{-1}
                             \left (
                                 r_\alpha (x,\xi,\cdo )
                                 \chi _{0,\xi}
                                 \right )(y)
      \]
      and use Parseval's formula to rewrite $c_{1,\alpha}$ into
      \begin{equation}\label{c1alphaReform}
                   c_{1,\alpha }(x,\xi ) = (2\pi)^{-d/2} \int  f_{\alpha} (x,\xi ,x-y)
                   D^{\alpha}_y
                   \left(
                     e^{i \psi(x,y,\xi)}
                     \chi(x,y)
                     a(y,\xi)
                   \right )\, dy.
       \end{equation}
By our choice of $\chi _0$ and $\varphi \in \Ph$ it follows that
for any multiindex $\beta$, on the support of the integrand of $c_{1,\alpha}$,
      \begin{multline}
      	\label{eq:3.20}
       | D^\beta_{\theta} r_\alpha (x,\xi,\theta)|
      \\
 \lesssim      
          \int_{0}^1  \omega _2(x,\varphi '_{x}(x,\xi) + t \theta)
          \norm{\varphi '_{x}(x,\xi) + t \theta}^{-(|\alpha|+|\beta|)} \cdot
          (1-t)^{M-1} \; t^{|\beta|}\, dt
      \\
 \lesssim
      	(\Theta _{2,\fy}\omega_2)(x,\xi)\,
          \norm{\xi}^{N_0-(|\alpha|+|\beta|)},
      \end{multline}
      for a suitable $N_0\in \mathbf Z_+$. In fact, $\omega$ is polynomially moderate, while
      the presence of ${\chi _0}$ in the integrand of $c_{1,\alpha}$ 
      and $t \in [0,1]$ imply
      $$
      |\theta|   \le \varepsilon_2 \norm{\xi},\quad
      |t \theta | \le \varepsilon_2 \norm{\xi}\quad \text{and}\quad 
      \norm{\varphi '_x (x,\xi)+t\theta} \asymp\norm{\xi}.
      $$
      
      \par
      
      We have also, for any multi-indices $\alpha,\beta$,
      \begin{multline}
      \label{eq:3.21}
          \left | 
            y ^\beta f_{\alpha} (x,\xi ,y)
          \right | =
          \left |
            \cF^{-1}
              \left (
                D^ \beta _{\theta}
                \left(
                  r_\alpha (x,\xi,\cdo )
                  \chi _{0,\xi} 
                \right )
              \right )
          \right |
      \\
       \lesssim
          |B_{\varepsilon_2\norm{\xi}}(0)| \cdot 
          \sup_{\theta \in B_{\varepsilon_2\norm{\xi}}(0)}
          \left|
                D^\beta_{\theta}
                \left(
                  r_\alpha (x,\xi,\theta)
                  \chi _{0,\xi} \left(
                                \frac{\theta}{\norm{\xi}}
                             \right)
                \right)
          \right| ,
      \end{multline}
      where $|B_{\varepsilon_2\norm{\xi}}(0)|$ is the volume of
      $B_{\varepsilon_2\norm{\xi}}(0)$. In view of \eqref{eq:3.20},
      \begin{multline}
      \label{eq:3.22}
         \left | D^\beta _{\theta}
                \left(
                  r_\alpha (x,\xi,\theta)
                  \chi _{0,\xi}(\theta )
                \right) \right | \lesssim
          \sum_{\gamma \le \beta}
            \left |
                D^\gamma_{\theta} r_\alpha (x,\xi,\theta)
            \right |
            \left |
                D^{\beta-\gamma}_{\theta}
                \chi _{0,\xi}
                \right |
      \\
      \lesssim
          \sum_{\gamma \le \beta}
            (\Theta _{2,\fy} \omega _2) (x,\xi) \norm{\xi}^{N_0-(|\alpha|+|\gamma|)}
            \norm{\xi}^{(|\gamma|-|\beta|)}
      \\
      \lesssim
          (\Theta _{2,\fy} \omega _2)(x,\xi) \norm{\xi}^{N_0-(|\alpha|+|\beta|)},
      \end{multline}
      %
      %
      Since $|B_{\varepsilon_2\norm{\xi}}(0)| \lesssim
      \norm{\xi}^{d}$, uniformly
      with respect to $\varepsilon_2\in(0,1)$, 
      \eqref{eq:3.20}, \eqref{eq:3.21}, and \eqref{eq:3.22} imply, for any
      multi-indices $\alpha,\beta$ and integer $N$,
      \begin{align*}
       \left| 
            y^\beta f_{\alpha} (x,\xi ,y)
          \right|
      &
          \lesssim
          (\Theta _{2,\fy} \omega _2)(x,\xi )\,\norm{\xi}^{d+N_0-|\alpha|-|\beta|}
      \intertext{giving that}
          \left | |y|^{N} \norm{\xi}^{N}  
            f_{\alpha} (x,\xi,y)
          \right |
          & \lesssim
         (\Theta _{2,\fy} \omega _2)(x,\xi)\,\norm{\xi}^{d+N_0-|\alpha |}.
          \end{align*}
      This in turn gives
      \[
          \left|  
            f_{\alpha} (x,\xi,y)
          \right|
          \lesssim
          \omega_{2,\varphi}(x,\xi)\,
          \norm{\xi}^{d+N_0-|\alpha|}\left( 1 + |y| \norm{\xi} \right)^{-N}.
      \]
      for any multi-index $\alpha$ and integer $N$.
      
      \par
      
      By letting $N = N_{1} + d+1$ with $N_{1}$ arbitrary integer,
      the previous estimates and \eqref{c1alphaReform} give
      \begin{multline*}
      |c_{1,\alpha }(x,\xi )|
      \\
      \lesssim  (\Theta _{2,\fy}\omega_2)(x,\xi)K_\alpha (x,\xi )\norm{\xi}^{d+N_0-|\alpha|}
         \cdot
         \int (1 + |y-x| \norm{\xi})^{-(d+1)}\,dy 
      \end{multline*}
where
      \begin{equation}\label{KalphaDef}
      K_\alpha (x,\xi ) :=  \sup_{y} \left (
                    \left|
                    (\cF _3^{-1}H_\alpha )(x,\xi ,y)
                    \right|
                    (1 + |y-x| \norm{\xi})^{-N_{1}}
                 \right ) .
      \end{equation}
      That is,
      \begin{equation}\label{eq:3.27}
      |c_{1,\alpha }(x,\xi )|
      \lesssim  (\Theta _{2,\fy}\omega_2)(x,\xi)K_\alpha (x,\xi )
      \norm{\xi}^{d+N_0-d-|\alpha|} .
      \end{equation}

      \par
      
       In order to estimate $K_\alpha (x,\xi )$, we
       notice that
      \begin{multline*}
      D_y^\alpha \left (
                                      e^{i \psi(x,y,\xi)}
                                      \chi(x,y)
                                      a(y,\xi)
                                   \right ) = 
      \\
      =   
      \sum_{\beta + \gamma + \delta = \alpha} 
           \frac{\alpha !}{\beta !\gamma !\delta !}
           \tau _{\beta}(x,y,\xi) e^{i \psi(x,y,\xi)}
           \;
           D_y^{\gamma} \chi(x,y)
           \;
           D_y^{\delta} a(y,\xi),
      \end{multline*}
      where $\tau _\beta$ are the same as in Lemma \ref{lemma:3.1}. Furthermore,
      by the support properties of $\chi$, Lemma \ref{lemma:3.2}
       shows that
       $$
       \norm x \asymp \norm y,\quad \text{and}\quad
       \omega _1(y,\xi )\lesssim \omega _1(x,\xi )\norm {y-x}^{M_1}
       $$
       in the support of $(\cF ^{-1}_3H_\alpha )(x,\xi ,y)$, for some constant
       $M_1\ge 0$. Hence, if $s=\min (r_1,1/2)$, we get
      \begin{multline*}
      |D_y^{\alpha} \left (
                                      e^{i \psi(x,y,\xi)}
                                      \chi(x,y)
                                      a(y,\xi)
                                   \right )|  
      \\
      \le 
      \sum_{\beta +\gamma +\delta =\alpha} 
           \frac{\alpha !}{\beta !\gamma !\delta !}
           |\tau _{\beta}(x,y,\xi) 
           \;
           D_y^{\gamma} \chi(x,y)
           \;
           D_y^{\delta} a(y,\xi)    |
      \\
         \lesssim   
      \sum _{\beta +\gamma +\delta =\alpha} 
           |\tau_{\beta}(x,y,\xi)|
           \;
           \norm{y}^{-|\gamma|}
           \;
           \omega _1(y,\xi)\,\norm{y}^{-r_1|\delta |}
      \\
       \lesssim 
      \omega _1(x,\xi)\norm {y-x}^{M_1}
      \sum_{\beta +\gamma +\delta =\alpha} \sum _j
           |y-x|^{|\delta _{j}|}
            \norm{x} ^{N_j-|\beta|}
           \norm{\xi} ^{N_j+|\theta_j|}
           \norm{y}^{- r_1|\gamma +\delta |}
      \\
       \lesssim 
      \omega _1(x,\xi)\norm {y-x}^{M_1}
      \sum _{\beta +\gamma +\delta =\alpha}\sum_j
           |y-x|^{|\delta _{j}|}
            \norm{x}^{-|\beta | /2}
           \norm{\xi}^{N_j+|\delta _j|}
           \norm{x}^{-r_1|\gamma +\delta |}
     \\
     \lesssim
      \omega_1(x,\xi)\norm {y-x}^{M_2}
      \sum _{\beta +\gamma +\delta =\alpha}\sum_j
           (|y-x|\norm{\xi})^{|\delta _{j}|}
           \norm{\xi}^{N_j}
           \norm{x}^{-s|\beta + \gamma + \delta |}
      \\
       \lesssim \omega _1(x,\xi )
           \norm{x}^{-s|\alpha |}
      \big \langle |x-y|\norm{\xi} \big \rangle ^{M_3}
      \sum_{\beta +\gamma +\delta =\alpha}\sum _j
           \norm {\xi}^{N_j},
      \end{multline*}
      %
      for some constants $M_2$ and $M_3$. Note that
      all terms in the last sum, are
      never larger than
      $
      \norm{\xi}^{|\beta | /2} \lesssim \norm{\xi}^{|\alpha | /2}
      $
      in view of \eqref{eq:3.4}. Moreover, \eqref{eq:3.3} implies
      \begin{multline*}
      	N_j \le N_j+\frac{1}{2}|\delta _j|=\frac{1}{2}(2N_j+|\delta _j|)
	\\
	 \le \frac{1}{2}\left(|\delta _j|+\sum_{k=1}^{N_j}|\beta _{jk}|\right)
	 =\frac{|\beta|}{2}\le\frac{|\alpha|}{2}.
      \end{multline*}

      \par
      
      We conclude that, for $N_1$ large enough, we get
      \begin{multline*}
      |c_{1,\alpha}(x,\xi )|
      \\[1ex]
      \lesssim
      \omega _1(x,\xi) (\Theta _{2,\fy }\omega_2)(x,\xi)
           \norm{x}^{-s|\alpha |}
           \norm{\xi}^{N_0-|\alpha | /2}
           \sup_{y\in\rr{d}}\big \langle |x-y|\norm{\xi} \big \rangle ^{M_3-N_1}
      \\
      \le
           \omega _1(x,\xi) (\Theta _{2,\fy }\omega_2)(x,\xi)
           \norm{x}^{-r_1|\alpha |/2}
           \norm{\xi}^{N_0-|\alpha | /2},
      \end{multline*}
      and \eqref{c1alphaEst} follows.

\par
      
Next we show that \eqref{c2alphaEst} holds. Let
      \begin{eqnarray}
      \nonumber
      \!\!\!\!\!\!\!\!\!\!
      \!\!
      f(x,y,\xi,\theta)
      \!\!\!
      & = &
      \!\!\!
      \scal{y}{\theta} - \psi(x,y,\xi)
      \\
      \label{eq:3.29}
      \!\!\!\!\!\!\!\!\!\!
      \!\!
      \!\!\!
      & = &
      \!\!\!
      \scal{y}{\theta} - (  \varphi(y,\xi) - \varphi(x,\xi) 
                        - \scal{y-x}{\varphi '_{x}(x,\xi)} ),
      \end{eqnarray}
      which implies
      \begin{equation*}
      f^\prime _y(x,y,\xi,\theta) 
       = 
      \theta - ( \varphi '_{y}(y,\xi) - \varphi '_{x}(x,\xi) ),
      \end{equation*}
      giving that
      \begin{equation*}
      \norm {f^\prime _y(x,y,\xi,\theta)}  \lesssim 
      \norm{\theta} + \norm{\xi}.
      \end{equation*}

      \par
      
      Let
      $$
      R_4= 
      \frac{1-\Delta_{\theta}}{\norm{x}^{2}}
      $$
      Then ${{^t} R_4}=R_4$ and
      $R_4 e^{i \scal{x}{\theta}} = e^{i \scal{x}{\theta}}$. By induction
      we get
      \begin{multline}\label{eq:3.31}
      c_{2,\alpha}(x,\xi ) =(2\pi)^{-d/2}
        \int  e^{i \scal{x}{\theta}}
               R_4^{l}
               \big (
                 r_\alpha (x,\xi,\cdo  )
                (
                    1 -
                    \chi _{0,\xi}
                ) \cdot
                H_\alpha (x,\xi ,\cdo )
                   \big )(\theta )
               d \theta
      \\
         =
         \sum _{j}
              \int e^{i \scal{x}{\theta}}
                 r_{j,\alpha} (x,\xi,\theta)
                    \chi _{j,\xi} (\theta )
                 D_\theta ^{\beta _j}H(x,\xi ,\theta )\, d\theta ,
      \end{multline}
      for every integer $l\ge 0$, where
      $\chi _{j,\xi }\equiv \chi _j(\cdo /\norm \xi )$, and $\chi _j$ and $r_{j,\alpha}$
      are smooth functions which satisfy
      \begin{equation}\label{eq:3.32}
      \chi _j\in L^\infty \cap C^\infty ,\quad \supp \chi _j
      \subseteq \rr d\setminus B_{\ep _2}(0)
      \end{equation}
      and
      \begin{equation}\label{eq:3.33}
      |r_{j,\alpha} (x,\xi,\theta)| \lesssim \omega_{2,\varphi}(x,\xi)\norm{\theta}^N
      \vartheta_{-2l_1,-|\alpha|}(x,\xi).
      \end{equation}
      Here $|\beta _j|\le 2l$ and the induction is done over $l\ge 0$.
      
      \par

We need to estimate the integrals in the sum \eqref{eq:3.31}. It is then convenient
to set
\begin{align}
g^j_{\beta ,\gamma ,\delta}(x,y,\xi) &\equiv \tau_{\beta}(x,y,\xi) \;
                                         \partial^{\gamma}_y \chi(x,y) \;
                                         y^{\beta_{j}} \partial^\delta_y a(y,\xi).
\intertext{and}
J^j_{\beta ,\gamma ,\delta}(x,\xi ,\theta ) &\equiv (2\pi)^{-d/2}\int e^{-if(x,y,\xi,\theta )}
       g^j_{\beta ,\gamma ,\delta}(x,y,\xi )\, d y .
\end{align}
In fact, by expanding
the Fourier transform in \eqref{eq:3.29}, and using the same notation as
in Lemma \ref{lemma:3.1}, we have that $c_{2,\alpha}$ is a (finite) linear combination
of
      \begin{multline}\label{eq:3.35}
      c_{2,j,\alpha}(x,\xi )
      \\
      = 
              \sum
              \frac{\alpha!}{\beta! \gamma! \delta!}
              \int 
                 e^{i\scal {x}{\theta}}
                 r_{j,\alpha} (x,\xi,\theta)
                 \chi_{j,\xi}(\theta )
                 J^j_{\beta ,\gamma ,\delta}(x,\xi ,\theta )\, d\theta,
       \end{multline}
where the sum is taken over all multi-indices $\beta ,\gamma ,\delta$ such
that $\beta +\gamma +\delta =\alpha$.

\par

In order to estimate $c_{2,j,\alpha}$ we first consider
$J^j_{\beta ,\gamma ,\delta}$ and the factor $g^j_{\beta ,\gamma ,\delta}$
in its integrand. 
By the relations
$$
\tau_{\beta} \in \SG^{0,0,|\beta|}_{1,1,1} \subseteq \SG^{0,0,|\alpha|}_{1,1,1},
\quad
\chi \in \SG^{0,0,0}_{1,1,1},
\quad
a \in \SG^{(\omega_0)}_{r_0,\rho_0},
$$
and $|\beta _j|\le 2l$, it follows that
$$
\partial ^{\gamma}_y \chi \in \SG^{0,-|\gamma|,0}_{1,1,1} 
               \subseteq \SG^{0,0,0}_{1,1,1}
\quad \text{and}\quad
y^{\beta_{j}} a(y,\xi) \in 
               \SG^{(\omega_0\cdot\vartheta_{2l_1,0})} _{\min\{r_0,1\},\rho_0}.
$$
This in turn gives
      \begin{equation}
      \label{eq:3.36}
      \begin{aligned}
      g^j_{\beta\gamma\delta} &\in \SG^{(\omega_3)}_{1,\min\{r_0,1\},\min\{\rho_0,1\}}, 
      \\
      \text{where}\quad \omega _3(x,y,\xi) &= \omega _0(y,\xi)
      \,\vartheta _{2l_1,|\alpha |}(y,\xi).
      \end{aligned}
      \end{equation}

\par

In order to estimate $|J^j_{\beta ,\gamma ,\delta}|$ we consider
the operator $R_3$ in \eqref{eq:3.10.1},
      which is admissible, since
      \begin{multline}
      \label{eq:3.37}
      |f^\prime _y(x,y,\xi,\theta)|   =   |\theta - (\varphi '_y (y,\xi) - \varphi '_{x}(x,\xi))|
      \\
                      \ge  |\theta| - |\varphi '_y (y,\xi) - \varphi '_{x}(x,\xi)|
                      \gtrsim   \norm{\theta} + \norm{\xi}
                      \asymp \norm {(\xi ,\theta )}
                     \gtrsim (\norm{\xi} \norm{\theta})^{\frac{1}{2}},
      \\
      \text{when}\quad
                      (x,y)\in \supp \chi
                     ,\ \theta \in \supp \chi _{j,\xi} ,
      \end{multline}
      provided $\varepsilon_1 \in (0,1)$ in
      the definition of $\chi$ is chosen small enough.
      
      \par
      
      In fact, if $\theta \in \supp \chi _{j,\xi}$, then 
      $|\theta| \ge \varepsilon_2\norm{\xi}/2$.
      Moreover, if $(x,y)\in \chi$, then $|y-x| \le \varepsilon_1 \norm{x}$,
      which gives
      \begin{multline*}
      \lefteqn{
                   \varphi^\prime_{x_j}(y,\xi) - 
                  \varphi^\prime_{x_j}(x,\xi) =
               }
      \\
      = \sum_{k=1}^d\int_{0}^{1} \varphi^{\prime\prime}_{x_j x_k}(x+t(y-x),\xi)
      (y_k -x_k)\, dt
      \\
      \lesssim \varepsilon_1 \norm{x} \norm{\xi}
            \int_{0}^{1} \norm{x+t(y-x)}^{-1} dt
      \lesssim \varepsilon_1 \norm{\xi} \norm{x} \norm{x}^{-1} = \varepsilon_1 \norm{\xi},
      \end{multline*}
and \eqref{eq:3.37} follows by straight-forward applications of these estimates.

\par

      We note that ${^t} R_3$ acts only on $g^j_{\beta\gamma\delta}$, 
      leaving $e^{i\scal{x}{\theta}}$, $r_{j,\alpha}$ and $\chi_{j,\xi}$
      unchanged. By applying \eqref{eq:3.10.2}, \eqref{eq:3.10.3},
      \eqref{eq:3.35} and \eqref{eq:3.36} we get, for any integer $l_0$,
      \begin{multline}\label{gIntegralFormulas}
      |J^j_{\beta ,\gamma ,\delta}(x,\xi ,\theta )|
      \\
      =
      \left | \int e^{-i f(x,y,\xi,\theta)} ({{^t} R_3})^{l_0} g^j_{\beta ,\gamma ,\delta}
      (x,y,\xi,\theta)\,dy \right |
      \\
      \le 
      \int  \frac{1}{|f^\prime _y(x,y,\xi ,\theta )|^{4l_0}} 
                                  \sum_{|\kappa| \le l_0}
                                  | P_{\kappa,l_0}
                                  \partial_y^{\kappa}
                                  g^j_{\beta\gamma\delta}(x,y,\xi,\theta)|\, dy .
      \end{multline}

\par

In the support of the latter integrands we have $|x-y|\le \ep _1\norm x$,
which gives
\begin{equation}\label{normx-yRel}
\begin{gathered}
\norm{x}\asymp \norm{y},\quad |x-y| \lesssim \norm{y},
\quad v(x-y)\lesssim \norm y^{m_0},
\\
v(y)\lesssim\norm{y}^{m_0}
\quad \text{and}\quad |y|\lesssim \norm x,
\end{gathered}
\end{equation}
for a suitable $m_0\in \mathbf Z _+$, which only depends
on $\omega _1$. Here $v\in \mascP (\rr d)$ is chosen such that
$\omega _1(x+y,\xi )\lesssim \omega _1(x,\xi )v(y)$. Hence
it suffices to evaluate the integrals in \eqref{gIntegralFormulas}
over the set
$$
\Omega = \sets {y\in \rr d}{|y|\le C_2\norm x|\ \text{and}\
C_1\norm x \le \norm y\le C_2\norm x},
$$
provided $C_1>0$ is small enough and $C_2>0$ is large enough.

\par

By brute-force computations, \eqref{eq:3.37}, \eqref{gIntegralFormulas},
\eqref{normx-yRel} and \eqref{gIntegralFormulas} we get
      \begin{multline}\label{gIntegralFormulas2}
      |J^j_{\beta ,\gamma ,\delta}(x,\xi ,\theta )|
      \\
      \lesssim 
      \norm \theta ^N \norm {(\xi ,\theta )}^{-4l_0} \sum _{|\kappa| \le l_0}
      \int _\Omega \omega _1(y,\xi )
      \vartheta  _{2l-|\kappa |,|\alpha |}(y,\xi )(\norm \theta +\norm \xi )^{3l_0}
      \norm y^{|\kappa |-l_0}\, dy
      \\
      \lesssim 
      \norm \theta ^N
      \norm {(\xi ,\theta )}^{-4l_0+3l_0} \omega _1(x,\xi )\vartheta 
      _{m_0+2l-l_0,|\alpha |}(x,\xi ) \sum _{|\kappa| \le l_0}
      \int _{|y|\le C\norm x} \, dy
      \\
      \lesssim
      \norm \xi ^{-l_0/2} \norm \theta ^{N-l_0/2} \omega _1(x,\xi )\vartheta 
      _{d+m_0+2l-l_0,|\alpha |}(x,\xi )
      \end{multline}

\par

Inserting this into \eqref{eq:3.35}, we get
      \begin{multline}\label{c2jalphaEst}
      |c_{2,j,\alpha}(x,\xi )|
      \lesssim (\Theta _{2,\fy} \omega _{2})(x,\xi)
      \vartheta_{-2l,-|\alpha|}(x,\xi)
      \sum 
      \int
      |J^j_{\beta ,\gamma ,\delta}(x,\xi ,\theta )|\, d\theta
      \\
      \lesssim \omega _1(x,\xi ) (\Theta _{2,\fy} \omega _{2})(x,\xi)
      \norm x^{d+m_0-l_0}
      \sum 
      \int \norm \xi ^{-l_0/2} \norm \theta ^{N-l_0/2}
      \, d\theta
      \\
      \lesssim
      \omega _1(x,\xi ) (\Theta _{2,\fy} \omega _{2})(x,\xi)
      \norm x^{-2l}\norm \xi ^{-l_0/2},      
      \end{multline}
      provided
      \[
      	l_{0} > \max\{ 2N, 2l+d+m_0 \} .
      \]
     Here the sums
     should be taken over all $j$ and $\beta$, $\gamma$ and $\delta$
     such that $\beta +\gamma +\delta =\alpha$.
%
	%
%
%
%
%
%
%
%
%
%
%
%
%

\par

Since $l$ and $l_0$ can be chosen arbitrarily large, and
\[
\omega_{2,\varphi}(x,\xi)\,\omega_0(x,\xi)\lesssim \norm{x}^m\norm{\xi}^\mu,
\]
for suitable $m,\mu\ge0$, it follows that \eqref{c2alphaEst}
is true for every integers $N_1$ and $N_2$. In particular it
follows that the hypothesis in \cite[Corollary 16]{CoTo} is fulfilled
with $a=h$ and $a_j$ being a suitable linear combination of
$c_\alpha$. This gives the result.
\end{proof}

\par

The next three theorems can be proved by modifying the arguments
given in \cite{coriasco},  similarly to the above proof of Theorem
\ref{thm:0.1}. The relations between Type I and Type II operators,
and the formulae for the formal-adjoints of the involved operators,
explained in Remark \ref{rem:sgsymm}, are useful in the
corresponding arguments.

\par

\begin{thm}
\label{thm:3.1}
Let $r_j,\rho _j\in [0,1]$, $\varphi \in \Ph$ and let $\omega _j
\in \mascP_{r_j,\rho _j}(\rr {2d})$, $j=0,1,2$, be such
that
$$
r_2=1,
\quad r_0=\min\{r_1,1\} ,\quad \rho _0=\min\{ \rho_1,\rho _2,1\},
\quad \omega _0 =\omega_1\cdot (\Theta _{1,\fy} \omega _2),
$$
and $\omega _2\in \mathscr{P}_{r,1}(\rr{2d})$ is
$(\phi,1)$-invariant with respect to $\phi \colon
x \mapsto \varphi^\prime_\xi (x,\xi)$.
Also let $a \in \SG^{(\omega _1)} _{r_1,\rho _1}(\rr {2d})$ and
$p \in \SG^{(\omega _2)}_{1,\rho _2}(\rr {2d})$.
Then
\begin{alignat*}{2}
\op _\varphi(a) \circ \op (p) &= \op_{\varphi}(c) \operatorname{Mod}
\op _\varphi (\SG ^{(\omega \vartheta _{-\infty ,0})}_0 ),& \quad \rho _1=0,
\\[1ex]
\op_\varphi(a) \circ \op (p) &= \op_{\varphi}(c) \operatorname{Mod}
\op (\mathscr S ),& \quad \rho _1>0,
\end{alignat*}
where the transpose ${^t}c$ of $c \in \SG^{(\omega _0)}
_{r_0,\rho _0}(\rr {2d})$ admits the asymptotic expansion
\eqref{eq:0.3}, after $p$ and $a$ have been replaced by
${^t}p$ and ${^t}a$, respectively.
\end{thm}

\par

\begin{thm}
\label{thm:3.2}
Let $r_j,\rho _j\in [0,1]$, $\varphi \in \Ph$ and let $\omega _j
\in \mascP_{r_j,\rho _j}(\rr {2d})$, $j=0,1,2$, be such
that
$$
\rho _2=1,
\quad r_0=\min\{r_1,r_2,1\} ,\quad \rho _0=\min\{ \rho_1,1\},
\quad \omega _0 =\omega_1\cdot (\Theta _{2,\fy} \omega _2),
$$
and $\omega _2\in \mathscr{P}_{r,1}(\rr{2d})$ is
$(\phi,2)$-invariant with respect to $\phi \colon
\xi \mapsto \varphi ^\prime _x(x,\xi)$.
Also let $b \in \SG^{(\omega _1)} _{r_1,\rho_1}(\rr {2d})$,
$p \in \SG^{(\omega _2)}_{r_2,1}(\rr {2d})$, $\psi$
be the same as in \eqref{eq:0.4}, and let $q \in \SG^{(\omega _2)}
_{r_2,1}(\rr {2d})$ be such that
\begin{equation}
\label{eq:1.11.2}
q(x,\xi)\sim\sum_{\alpha}\frac{i^{|\alpha|}}{\alpha!}D^\alpha_x D^\alpha_\xi\overline{p(x,\xi)}.
\end{equation}
Then
\begin{alignat*}{2}
\op _\varphi ^*(b) \circ \op(p) &= \op_{\varphi}(c) \operatorname{Mod}
\op _\varphi ^*(\SG ^{(\omega \vartheta _{0,-\infty })}_0 ),& \quad r_1=0,
\\[1ex]
\op _\varphi ^*(b) \circ \op(p) &= \op_{\varphi}(c) \operatorname{Mod}
\op (\mathscr S ),& \quad r_1>0,
\end{alignat*}
where $c \in \SG^{(\omega _0)}_{r_0,\rho _0}(\rr {2d})$
admits the asymptotic expansion
\begin{equation}
\label{eq:3.38bis}
c(x,\xi) \sim \sum_{\alpha} \frac{i^{|\alpha|}}{\alpha!} 
(D^\alpha_\xi q)(x, \varphi^\prime_x(x,\xi))
D^\alpha_y \!\!\left[ e^{i \psi(x,y,\xi)} b(y,\xi) \right ]_{y=x}.
\end{equation}
\end{thm}

\par

\begin{thm}
\label{thm:3.3}
Let $r_j,\rho _j\in [0,1]$, $\varphi \in \Ph$ and let $\omega _j
\in \mascP_{r_j,\rho _j}(\rr {2d})$, $j=0,1,2$, be such
that
$$
r_2=1,
\quad r_0=\min\{r_1,1\} ,\quad \rho _0=\min\{ \rho_1,\rho _2,1\},
\quad \omega _0 =\omega _1\cdot (\Theta _{1,\fy} \omega _2),
$$
and $\omega _2\in \mathscr{P}_{r,1}(\rr{2d})$ is
$(\phi,1)$-invariant with respect to $\phi \colon
x \mapsto \varphi^\prime_\xi (x,\xi)$.
Also let $a \in \SG^{(\omega _1)} _{r_1,\rho _1}(\rr {2d})$ and
$p \in \SG^{(\omega _2)}_{1,\rho _2}(\rr {2d})$.
Then
\begin{alignat*}{2}
\op (p)\circ \op _\varphi ^*(b) &= \op_{\varphi}(c) \operatorname{Mod}
\op _\varphi ^*(\SG ^{(\omega \vartheta _{-\infty ,0})}_0 ),& \quad \rho _1=0,
\\[1ex]
\op (p)\circ \op _\varphi ^*(b) &= \op_{\varphi}(c) \operatorname{Mod}
\op (\mathscr S ),& \quad \rho _1>0,
\end{alignat*}
where the transpose ${^t}c$ of $c \in \SG^{(\omega _0)}
_{r_0,\rho _0}(\rr {2d})$ admits the asymptotic expansion
\eqref{eq:3.38bis}, after $q$ and $b$ have been replaced by
${^t}q$ and ${^t}b$, respectively.
\end{thm}

\par

\subsection{Composition between $\SG$ FIOs of type I and type II}
\label{subs:2.4}
The subsequent Theorems \ref{thm:3.4} and \ref{thm:3.5} deal with the
composition of a type I operator with a type II operator, and show that such compositions
are pseudo-differential operators with symbols in natural classes. We give the argument
only for Theorem \ref{thm:3.4}, since the proof of Theorem \ref{thm:3.5} follows, with
similar modifications, from the one given in \cite{coriasco} for the corresponding
composition result.

\par

The main difference, with respect to the arguments in \cite{coriasco}
for the analogous composition results, is that we again make use, in both cases, of 
the generalized asymptotic expansions introduced in Definition \ref{def:gensgasexp}.
This allows to overcome the additional difficulty, not arising there, that the amplitudes
appearing in the computations below involve weights which are still polynomially bounded, but
which do not satisfy, in general, the moderateness condition \eqref{moderate}. On the 
other hand, all the terms appearing in the associated asymptotic expansions belong 
to $\SG$ classes with weights of the form $\widetilde{\omega}_{2,\varphi}\cdot\vartheta_{-k,-k}$,
where $\widetilde{\omega}=\omega_1\cdot\omega_2$. In view of the results
in \cite{CoTo}, this allows to conclude as desired, since the
remainders are of the  forms given in Proposition \ref{propasymp}.

\par

In order to formulate our next result, it is convenient to let $S_\fy$ with $\fy \in \Ph$,
be the operator, defined by the formulas
\begin{equation}\label{SvarphiDef}
\begin{gathered}
(S_\fy f)(x,y,\xi ) = f(x,y,\Phi (x,y,\xi ))\cdot \left | \det \Phi '_\xi (x,y,\xi )\right |
\\[1ex]
\text{where}\quad
\int _0^1 \fy _x'(y+t(x-y),\Phi (x,y,\xi ))\, dt =\xi .
\end{gathered}
\end{equation}
That is, for every fixed $x,y\in \rr d$, $\xi \mapsto \Phi (x,y,\xi )$ is the inverse of the map
\begin{equation}\label{eq:diffeo}
\xi \mapsto \int _0^1 \fy _x'(y+t(x-y),\xi )\, dt .
\end{equation}
Notice that, as proved in \cite{coriasco}, the map \eqref{eq:diffeo} is indeed invertible for $(x,y)$ belonging to the 
support of the elements of $\Xi ^\Delta (\ep )$, provided $\ep$ is chosen suitably small, and it turns out to be, in that case, a $\SG$ diffeomorphism with $\SG^0$ parameter dependence.

\par

\begin{thm}
\label{thm:3.4}
Let $r_j\in [0,1]$, $\varphi \in \Ph$ and let $\omega _j
\in \mascP_{r_j,1}(\rr {2d})$, $j=0,1,2$, be such that $\omega _1$ and
$\omega _2$ are $(\phi,2)$-invariant
with respect to $\phi \colon
\xi \mapsto (\varphi ^\prime _x)^{-1}(x,\xi)$,
$$
\quad r_0=\min\{r_1,r_2,1\} \quad \text{and}
\quad \omega _0(x,\xi ) =\omega_1(x,\phi (x,\xi ))\omega_{2}(x,\phi (x,\xi )),
$$
Also let $a \in \SG^{(\omega _1)} _{r_1,1}(\rr {2d})$ and
$b \in \SG^{(\omega _2)}_{r_2,1}(\rr {2d})$.
Then
\begin{equation*}
\op _\varphi (a) \circ \op _\varphi ^*(b) = \op (c),
\end{equation*}
for some $c \in \SG^{(\omega _0)}_{r_0,1}(\rr {2d})$.
Furthermore, if $\ep \in (0,1)$, $\chi \in \Xi ^\Delta (\ep )$,
$c_0(x,y,\xi )= a(x,\xi )b(y,\xi )\chi (x,y)$ and $S_\fy$ is given by
\eqref{SvarphiDef}, then $c$ admits the asymptotic expansion
\begin{equation}
\label{eq:3.38}
c(x,\xi) \sim \sum _{\alpha} \frac{i^{|\alpha |}}{\alpha !} 
(D^\alpha _yD^\alpha_\xi (S_\fy c_0))(x,y,\xi) \big |_{y=x}.
\end{equation}
\end{thm}

\par

\begin{proof}
Let us write explicitly the composition for $u \in \cS$. We find
\begin{align*}
\op_\varphi(a)\circ&\op^*_\varphi(b) u(x)
\\
& \!=\!
      (2\pi)^{-d}\int e^{i \varphi(x,\xi)} \, a(x,\xi)
      \left[\int e^{-i \varphi(y,\xi)} \, \overline{b(y,\xi)} \, u(y) \, dy\right] \!d \xi 
\\
& \!=\!
      (2\pi)^{-d}\int  e^{i f(x,y,\xi)} \, q(x,y,\xi) \, u(y) \, dy d \xi, 
\end{align*}
where we have set $f(x,y,\xi) = \varphi(x,\xi) - \varphi(y,\xi)$
and $q(x,y,\xi) = a(x,\xi)\cdot\overline{b(y,\xi)} \in \SG^{(\omega)}_{r_1,r_2,1}$. Let us
choose $\chi \in \Xi^\Delta(\varepsilon)$, $\varepsilon\in(0,1)$, and write
\begin{eqnarray*}
\op_\varphi(a)\circ\op^*_\varphi(b) u(x)
& = & 
      (2\pi)^{-d}\int e^{i f(x,y,\xi)} \, c_{0}(x,y,\xi) \, u(y) \, dy d \xi 
\\
& + &
      (2\pi)^{-d}\int e^{i f(x,y,\xi)} \, c_{1}(x,y,\xi) \, u(y) dy d\xi
\\
& = & (C_{0} + C_{1}) u(x)
\end{eqnarray*}
with $c_{0}(x,y,\xi) = \chi(x,y) q(x,y,\xi)$ and 
$c_{1}(x,y,\xi)=(1-\chi(x,y))q(x,y,\xi)$. Of course,
$c_0,c_1\in\SG^{(\omega)}_{r,s,1}$. We begin by proving that, under our hypotheses, 
$C_{1}$ is a smoothing operator. Then we will show that
$C_0$ can be rewritten as the $\SG$ pseudo-differential operator described in the statement, 
provided $\varepsilon\in(0,1)$ is chosen suitably small.

\begin{enumerate}

\item \textbf{$C_{1}$ is smoothing.} \newline
      First of all, notice that we have 
      $|x-y| \ge \frac{\varepsilon}{2} \norm{x}$ on $\supp{c_{1}}$. Then, 
      in the integral defining $C_1 u(x)$, we can use the operator 
      \[
      	R_3=\frac{1}{|f^\prime_\xi(x,y,\xi)|^{2}}
      \sum_{k=1}^n f^\prime_{\xi_j}(x,y,\xi) \, D_{\xi_j},
      \]
      analogous to that defined in \eqref{eq:3.10.1}.
      In fact, let us set $v = \varphi '_{\xi} (x,\xi)$ and 
      $w = \varphi '_{\xi} (y,\xi)$. By making use of Proposition 
      \ref{prop:3.2} and in view of $\varphi \in \SG^{1,1}_{1,1}$, we can write,
      for a suitable constant $M>0$,
      \begin{eqnarray*}
      |x - y| &  =  & |  (\varphi '_{\xi} )^{-1}(v,\xi) 
                       - (\varphi '_{\xi} )^{-1}(w,\xi)|
      \\
              &  =  & \left|
                         \int_{0}^{1}  \scal{v - w}{d_{x}(\varphi '_{\xi} )^{-1}(tv + (1-t)w,\xi)}\,dt 
                      \right|
      \\
              & \le & |v-w| \sup_{\rr{d} \times \rr{d}}
                            \|d_{x}(\varphi '_{\xi} )^{-1}
                                (z,\xi)  \|
      \\
              & \le & M |  \varphi '_{\xi} (x,\xi)
                         - \varphi '_{\xi} (y,\xi)| 
      \\
              &  =  & M |  f^\prime_\xi(x,y,\xi)|,
      \end{eqnarray*}
      which implies
      \[
      		|f^\prime_\xi(x,y,\xi)| \gtrsim |x-y| \gtrsim \norm{x} + \norm{y}
      \]
      on $\supp{c_1}$. Then, using $R_3 e^{if}=e^{if}$,
      \eqref{eq:3.10.2}, \eqref{eq:3.10.3}, \eqref{eq:3.10.4}
      and again $\varphi \in \SG^{1,1}_{1,1}$, for any integer $l$,
      \[
      C_{1}u(x) = (2\pi)^{-d}\int  e^{if(x,y,\xi)} \, 
                                           (({{^t} R_3})^l c_{1})(x,y,\xi) \,
                                           u(y)\,dyd \xi
      \]
      and
      \[  
      \begin{aligned}
      (({^t R_3})^l c_{1}) (x,y,\xi)
                &   =     \frac{1}{|f^\prime_\xi(x,y,\xi)|^{4l}}
                           \sum_{|\alpha| \le l}
                           P_{l\alpha}
                           \partial_{\xi}^\alpha c_{1} (x,y,\xi)
      \\       
                & \lesssim   \frac{\sum_{|\alpha| \le l}
                                 (\norm{x} + \norm{y})^{3l}
                                 \norm{\xi}^{|\alpha|-l}
                                 \omega_1(x,\xi)\,\omega_2(y,\xi)
                                 \norm{\xi}^{-|\alpha|}
                                }{(\norm{x} + \norm{y})^{4l}}
      \\       
                &   \lesssim    \frac{
                	               \omega_1(x,y,\xi)
                                }{(\norm{x} + \norm{y})^l}
                                \norm{\xi}^{-l}
      \end{aligned}
      \]                          
      Then, we can rewrite $C_1$ as
      \begin{eqnarray*}
      C_{1}u(x) & = & (2\pi)^{-d} \int \left[
       \int e^{i f(x,y,\xi)} \, 
                                           ({^t R_3})^l c_{1}(x,y,\xi) \, d \xi \; 
                                           \right] u(y)\,dy      \\                                  
                & = & \int k_1(x,y)\, u(y)\,dy,
      \end{eqnarray*}
      with an arbitrarily chosen large integer $l$. Recalling that $\omega$, by \eqref{moderate},
      is polynomially bounded, and  
      $\norm{x} + \norm{y} \ge (\norm{x} \norm{y})^{\frac{1}{2}}$, 
      it follows $k_1(x,y)\lesssim (\norm{x}\norm{y})^{-N}$ for any integer $N$.
      The estimates for the derivatives of $D^\alpha_x D^\beta_y k_1(x,y)$ follow
      similarly by differentiation
      under the integral sign, since then we just have to start with some
      other $\tilde{c}_1\in\SG^{(\tilde{\tilde{\omega}})}_{r,s,1}$.
        
\item \textbf{$C_{0}$ is a generalized $\SG$ pseudo-differential operator.} \newline
      On $\supp{c_{0}}$ we have $|x-y|\le\varepsilon\norm{x}\Rightarrow 
      \norm{x}  \asymp \norm{y}$. Let us define,
      \[
      \tilde{d}_{x} \varphi(x,y,\xi) = \int_{0}^{1}  
      \varphi '_x (y+t(x-y),\xi)\,dt.
      \]
      In \cite{coriasco} it has been proved that 
      \[
      	\phi\colon\rr{d}\times\rr{2d}:(\xi,(x,y))\mapsto\phi(\xi,x,y)=\tilde{d}_{x} \varphi(x,y,\xi)
      \]
      is, on the support of $c_0$,
      an $\SG$ diffeomorphism with $\SG^0$ parameter dependence. For the sake of completeness,
      we recall the proof of this result. First observe
      \begin{align*}
      \tilde{d}_{x} \varphi(x,y,\xi)
                    & = \varphi^\prime_x (y,\xi)+w(x,y,\xi),
      \\
       w(x,y,\xi)   
                    & = \int_{0}^{1} \int _{0}^{1}
                          (x-y)\cdot H(y+t_{1}t_{2}(x-y),\xi) dt_{1} dt_{2}, 
      \\
        H(x,\xi)
        	            & = \varphi^{\prime\prime}_{xx}(x,\xi),
      \end{align*}
      \begin{equation}
      \label{eq:3.46}
      \begin{aligned}
      \Rightarrow
        d _ \xi
        \tilde{d}_{x} \varphi(x,y;\xi)
                    &= \varphi^{\prime\prime}_{x\xi} (y,\xi)
      \\
      &+ \int_{0}^{1} \int _{0}^{1}
                           \; t_{1}
                          (x-y)\cdot
                           H^\prime_\xi
                           (y+t_{1}t_{2}(x-y),\xi) \,dt_1dt_2.
      \end{aligned}
      \end{equation}
      Provided $\varepsilon\in(0,1)$ is small enough, the integrand in \eqref{eq:3.46} can be estimated 
      on $\supp c_0$ as follows:
      \begin{eqnarray*}
      & & 
      \sum_{k=1}^d(x_k - y_k) \, \partial_{\xi _l}
                     \varphi^{\prime\prime}_{x_j x_{k}}(y+t_{1}t_{2}(x-y),\xi) 
      \\
      & &
      \lesssim
      |x-y| \sup_{t \in [0,1]} 
            \norm{y+t(x-y)}^{-1}
      \\
      & &
      \lesssim
      \varepsilon \norm{x} \norm{y}^{-1} \lesssim \varepsilon,
      \end{eqnarray*}
      so that the Jacobian of $\tilde{d}_{x}\varphi(x,y,\xi)$ is a 
      small perturbation of the one of $\varphi '_x (y,\xi)$. Then, possibly
      taking a smaller value of $\varepsilon$ and recalling $\varphi \in \Phr$,
      on $\supp c_0$ we can assume
      \[
      \left|
        \det
        d_\xi 
        \tilde{d}_{x} \varphi(x,y,\xi)
      \right| \ge \frac{\kappa}{2} > 0.
      \]
      Moreover, it is easy to see that, on $\supp c_0$, the components
      of $\tilde{d}_{x} \varphi(x,y,\xi)$ satisfy $\SG^{0,0,1}_{1,1,1}$ estimates, since
      \begin{equation}
      \label{eq:3.48}
      	  \begin{aligned}
         \partial^{\alpha}_x \partial^{\beta}_y \partial_{\xi}^\gamma 
          &\tilde{d}_{x} \varphi(x,y,\xi) \lesssim
          \\
          &\lesssim  \norm{y}^{-|\alpha+\beta|}\norm{\xi}^{1-|\gamma|}
            =  \norm{x}^{-|\alpha|}
            \norm{y}^{-|\beta|}
            \norm{\xi}^{1-|\gamma|} .
            \end{aligned}
      \end{equation}
      We now prove that, on $\supp c_0$,
      \[
      \norm{\tilde{d}_{x} 
      \varphi(x,y,\xi)}  \asymp \norm{\xi}.
      \]
      In fact, the upper bound is immediate, and we also have
      \begin{align*}
      |w(x,y,\xi)| & \le |x-y|\cdot\sup_{t\in[0,1]}\|H(y+t(x-y),\xi)\|\lesssim\varepsilon\norm{x}\norm{y}^{-1}\norm{\xi}
      \\
      &\lesssim \varepsilon\norm{\xi}\lesssim\varepsilon\norm{\varphi^\prime_x(y,\xi)}
      \\
         \norm{\tilde{d}_{x} \varphi(x,y,\xi)} & = \norm{\varphi^\prime_x(y,\xi)+w(x,y,\xi)}
          \asymp\norm{\varphi^\prime_x(y,\xi)} \asymp\norm{\xi}.
      \end{align*}
      Then, with a suitable choice of $\varepsilon\in(0,1)$,
      $\tilde{d}_{x} \varphi(x,y;\xi)$ satisfies all the requirements 
      of Definition \ref{def:sgdiffeo}, and, on $\supp c_0$, 
      $\tilde{d}_{x} \varphi(x,y,\xi)$ is an $\SG$ diffeomorphism with
      $\SG^{0}$ parameter dependence. With this in mind, we can rewrite
      $C_{0}u(x)$ as
      \begin{eqnarray*}
         C_{0}u(x) & = & (2\pi)^{-d}\int 
                              e^{i( \varphi(x,\xi) - \varphi(y,\xi) )} \,
                              c_{0}(x,y,\xi) \, u(y)\,d \xi dy
      \\
      & = & (2\pi)^{-d}\int 
                 e^{i \scal{x-y}{\tilde{d}_{x} \varphi(x,y,\eta)}} \,
                 c_{0}(x,y,\eta) \, u(y)\,d \eta dy.
      \end{eqnarray*}
      By the above arguments, it follows that we can make the substitution
      \[
      	\xi  = \tilde{d}_{x} \varphi(x,y;\eta)
      	\Leftrightarrow
      	\eta = (\tilde{d}_{x} \varphi)^{-1}(x,y;\xi),
      \]
      so that we can conclude
      \begin{eqnarray*}
      C_{0}u(x) & = & (2\pi)^{-d}\int e^{i \scal{x-y}{\xi}}
                      \, c_{0}(x,y, (\tilde{d}_{x} \varphi)^{-1}(x,y,\xi))\cdot 
      \\
                &   & \cdot \left|
                         \det d_\xi
                              (\tilde{d}_{x} \varphi)^{-1}(x,y,\xi)
                      \right| 
                      u(y)\,dy d \xi
      \\
                & = & (2\pi)^{-d} \int e^{i \scal{x-y}{\xi}}
                      \, (S_\varphi c_0)(x,y,\xi) \, u(y)\,dy d \xi,
      \end{eqnarray*}
      where, by the definition of $S_\fy$,
      \begin{equation}\label{eq:3.42}
      	(S_\varphi c_0)(x,y,\xi)
	=
	c_{0}(x,y, (\tilde{d}_{x} \varphi)^{-1}(x,y,\xi))
       \cdot \left|
                         \det d_\xi
                              (\tilde{d}_{x} \varphi)^{-1}(x,y,\xi)
                      \right| 
      \end{equation}
\end{enumerate}
The pseudo-differential operator with amplitude $S_\varphi c_0$ can be rewritten as
an operator with symbol through the asymptotic expansion
\begin{equation}\label{eq:asympt}
c(x,\xi)\sim
\sum_{\alpha} \frac{i^{|\alpha|}}{\alpha!}\left.(D^\alpha_y D^\alpha_\xi (S_\varphi c_0))(x,y,\xi) \right|_{y=x}.
\end{equation}
That \eqref{eq:asympt} is indeed an expansion as in Definition \ref{def:gensgasexp}
and Proposition \ref{propasymp} is a consequence of the following observations:
\begin{itemize}
	\item[-] the second factor in \eqref{eq:3.42} is an amplitude in $\SG^{0,0,1}
	_{1,1,1}$, see \cite{coriasco};
	\item[-] as a weight on $\rr{3d}$, $\omega$ could happen not to be polynomially
	moderate, but it is still polynomially
	bounded, that is, $\omega(x,y,\xi)\lesssim\norm{x}^{m_1}\norm{y}^{m_2}
	\norm{\xi}^\mu$ on $\rr{3d}$, for suitable
	$m_1,m_2,\mu\ge0$, and the same clealry holds for $\omega(x,y,(\tilde{d}_{x} \varphi)^{-1}(x,y,\xi))$; 
	then, $\Op{S_\fy c_0}$ still gives rise to a pseudo-differential operator of $\SG$ type, see \cite{Co},
	and in the asymptotic expansion argument it is still possible to obtain remainders
	of arbitrary low order in at least one of the
	variables, giving rise to remainders of the type in Proposition \ref{propasymp};
	\item[-] the $\alpha$-derivatives of the first factor in \eqref{eq:3.42} with respect to $y$ and $\xi$, evaluated for
	$y=x$, contain only the derivatives of the $\SG$ diffeomorphism with $\SG^0$ parameter dependence
	 $\phi\colon(\xi,x)\mapsto(\varphi '_x)^{-1} (x,\xi)$ and the derivatives of $a$ and $b$ evaluated at the image
	 $(x,(\varphi '_x)^{-1} (x,\xi))$; then, in view of the properties of the $\SG$ diffeomorphism $\phi$ and
	 the $(\phi,1)$-invariance of $\omega_1$ and $\omega_2$, Lemma \ref{lemma:omegainv} implies
	 that $\tilde{\omega}$ is again a 
	 polynomially moderate weight, and ``the order of the terms decreases'' (at least with respect to the covariable).
\end{itemize}
By Proposition \ref{propasymp} and the results in \cite{CoTo},
we then find $c \in \SG^{(\widetilde{\omega})}_{\tilde{r},1}$, with $\widetilde{\omega}$ and $\tilde{r}$
as stated, satisfying $\op(S_\varphi c_0)=\op(c)$ modulo remainders. The proof is complete.
\end{proof}

\par

For the next result it is convenient to modify the operator $S_\fy$ in \eqref{SvarphiDef}
such that it fulfills the formulas
\begin{equation}\label{SvarphiDef2}
\begin{gathered}
(S_\fy f)(x,\xi ,\eta ) = f(\Phi (x,y,\xi ),\xi,\eta )\cdot \left | \det \Phi '_x (x,\xi ,\eta )\right |
\\[1ex]
\text{where}\quad
\int _0^1 \fy _\xi '(\Phi (x,\xi ,\eta ), \eta +t(\xi -\eta ),)\, dt =x .
\end{gathered}
\end{equation}

\par

\begin{thm}
\label{thm:3.5}
Let $\rho _j\in [0,1]$, $\varphi \in \Phr$ and let $\omega _j
\in \mascP_{1,\rho _j}(\rr {2d})$, $j=0,1,2$, be such that $\omega _1$ and
$\omega _2$ are $(\phi,1)$-invariant
with respect to $\phi \colon
x \mapsto (\varphi ^\prime _\xi )^{-1}(x,\xi)$,
$$
\quad \rho _0=\min\{ \rho _1,\rho _2,1\} \quad \text{and}
\quad \omega _0(x,\xi ) =\omega_1(\phi (x,\xi ),\xi )\omega_{2}(\phi (x,\xi ),\xi ),
$$
Also let $a \in \SG^{(\omega _1)} _{1,\rho _1}(\rr {2d})$ and
$b \in \SG^{(\omega _2)}_{1,\rho _2}(\rr {2d})$.
Then
\begin{equation*}
\op _\varphi ^*(b) \circ \op _\varphi (a) = \op (c),
\end{equation*}
for some $c \in \SG^{(\omega _0)}_{1,\rho _0}(\rr {2d})$.
Furthermore, if $\ep \in (0,1)$, $\chi \in \Xi ^\Delta (\ep )$,
$c_0(x,\xi ,\eta )= a(x,\xi )b(x,\eta )\chi (\xi ,\eta )$ and $S_\fy$ is given by
\eqref{SvarphiDef2}, then $c$ admits the asymptotic expansion
\begin{equation}
\label{eq:3.38B}
c(x,\xi) \sim \sum _{\alpha} \frac{i^{|\alpha |}}{\alpha !} 
(D^\alpha _xD^\alpha_\eta (S_\fy c_0))(x,\xi ,\eta ) \big |_{\eta =\xi}.
\end{equation}
\end{thm}

\par

\subsection{Elliptic FIOs of generalized $\SG$ type and parametrices.
Egorov Theorem}
\label{sec:2.5}
The results about the parametrices of the subclass of
generalized ($\SG$) elliptic  Fourier integral operators
are achieved in the usual way, by means of the composition
theorems in Subsections \ref{subs:2.3} and \ref{subs:2.4}.
The same holds for the versions of the Egorov's theorem
adapted to the present situation. The additional conditions,
compared with the statements in \cite{coriasco}, concern
the invariance of the weights, so that the hypotheses of
the composition theorems above are fulfilled. Here we omit
the proofs.

\begin{defn}
\label{def:3.1}
A type I  or a type II $\SG$ FIO, $\op_{\varphi}(a)$ or $\op^*_{\varphi}(b)$, respectively,
is said ($\SG$) elliptic if $\varphi \in \Phr$ and the amplitude $a$, respectively $b$, is ($\SG$) elliptic. 
\end{defn}
\begin{lemma}
\label{lemma:3.15}
Let a type I $\SG$ FIO $\op_{\varphi}(a)$ be elliptic, with 
$a\in\SG^{(\omega)}_{1,1}(\rr{2d})$. Assume that $\omega$ is $\phi$-invariant, 
$\phi=(\phi_1,\phi_2)$, where $\phi_2$ and $\phi_1$ are the $\SG$ diffeomorphisms
appearing in Theorems \ref{thm:3.4} and \ref{thm:3.5}, respectively.
Then, the two pseudo-differential operators $\op_{\varphi}(a)\circ\op^*_{\varphi}(a)$ and $\op^*_{\varphi}(a)\circ\op_{\varphi}(a)$ are $\SG$ elliptic.
\end{lemma}

\begin{thm}
\label{thm:3.8}
Let $\varphi\in\Phr$, $a\in\SG^{(\omega)}_{1,1}(\rr{2d})$, with $a$ $\SG$ elliptic.
Assume that $\omega$ is $\phi$ invariant, 
$\phi=(\phi_1,\phi_2)$, where $\phi_2$ and $\phi_1$ are the $\SG$ diffeomorphisms
appearing in the Theorems \ref{thm:3.4} and \ref{thm:3.5}, respectively.
Then, the elliptic $\SG$ FIOs $\op_{\varphi}(a)$ and $\op^*_{\varphi}(a)$ admit a parametrix.
These are elliptic $\SG$ FIOs of type II and type I, respectively.
\end{thm}
%

\noindent
As usual, in the next two results we need
the canonical transformation $\phi\colon(x,\xi)\mapsto(y,\eta)$ generated by
the phase function $\varphi$, namely
\begin{equation}\label{eq:phi}
	\begin{cases}
		\xi=\varphi^\prime_x(x,\eta)
		\\
		y=\varphi^\prime_\xi(x,\eta).
	\end{cases}
\end{equation}

\begin{thm}
\label{thm:3.21}
Let $A = \op_{\varphi}(a)$ be an $\SG$ FIO of type I with
$a \in \SG^{(\omega_0)}_{1,1}(\rr{2d})$ and $P = \Op{p}$  a pseudo-differential operator with
$p \in \SG^{(\omega)}_{1,1}(\rr{2d})$. Assume that $\omega$ is $\phi$-invariant,
where $\phi$ is the canonical transformation \eqref{eq:phi}, associated with $\varphi$. Assume also that $\omega_0$ is $(\tilde{\phi},2)$-invariant, where $\tilde{\phi}\colon\xi\mapsto(\varphi '_{x})^{-1}(x,\xi)$.
Then, setting $\eta = (\varphi '_{x})^{-1}(x,\xi)$ we have
\begin{equation}
\label{eq:3.72.3bis}
\begin{aligned}
\Sym{A\circ P\circ A^{*}}(x,\xi) &= p( \varphi^\prime_{\xi}(x, \eta), \eta)\,|a(x,\eta)|^2\,
|\det \varphi^{\prime\prime}_{x\xi}(x,\eta)|^{-1}
\\
& \mod \SG^{(\widetilde{\omega}\cdot\vartheta_{-1,-1})}_{1,1}(\rr{2d}),
\end{aligned}
\end{equation}
which is an element of $\SG^{(\widetilde{\omega})}_{1,1}(\rr{2d})$ with 
\[
	\widetilde{\omega}(x,\xi)=\omega(\phi(x,\xi))\cdot
	\omega_0(x,(\varphi '_{x})^{-1}(x,\xi))^2.
\]
\end{thm}
\begin{thm}
\label{thm:3.21ell}
Let $A = \op_{\varphi}(a)$ be an elliptic $\SG$ FIO of type I with
$a \in \SG^{(\omega_0)}_{1,1}(\rr{2d})$ and $P = \Op{p}$  a pseudo-differential operator with
$p \in \SG^{(\omega)}_{1,1}(\rr{2d})$. Assume that $\omega$ is $\phi$-invariant,
where $\phi$ is the canonical transformation \eqref{eq:phi}, associated with 
$\varphi$. Then,  we have
\begin{equation}
\label{eq:3.72.3}
\Sym{A\circ P\circ A^{-1}}(x,\xi) = p( \phi(x,\xi)) \!\!\mod 
\SG^{(\widetilde{\omega}\cdot\vartheta_{-1,-1})}_{1,1}(\rr{d}),
\end{equation}
with $\widetilde{\omega}(x,\xi)=\omega(\phi(x,\xi))$.
\end{thm}

\par

\section{$L^2(\rr{d})$-continuity of regular generalized $\SG$
FIOs with uniformly bounded amplitude}\label{sec3}

\par

In this section we prove a $L^2(\rr{d})$-boundedness results
for the generalized $\SG$ FIOs with 
amplitude $a  \in \SG^{0,0}_{r,\rho}(\rr{2d})$, $r,\rho\ge0$,
and regular phase function.
More precisely, we have the following.

\par

\begin{thm}
\label{thm:3.6}
Let $A=\op_{\varphi}(a)$ be a type I $\SG$ Fourier integral operator
with $\varphi \in \Phr$ and $a  \in \SG^{0,0}_{r,\rho}(\rr{2d})$, $r,\rho\ge0$.
Then, $A \in \cL(L^{2}(\rr{d}))$.
\end{thm}

\par

The proof of Theorem \ref{thm:3.6} is given as an adapted
version of a general $L^{2}$-boundedness result by Asada and Fujiwara
\cite{AsFu78}. The argument below is a slightly modified version of the one originally given
in \cite{coriasco} for the case $a\in\SG ^{0,0}_{1,1}$
(see also, e.g., \cite{BoBuRo} and \cite{RuSu}, and the references quoted therein).
We illustrate below the full argument, since Theorem \ref{thm:3.6} is a first relevant
mapping property for the class of 
generalized $\SG$ Fourier integral operators. Other mapping properties of this type,
including a continuity result between suitable weighted modulation spaces,
are given in \cite{CJT4}.

\par

We need some preparations for the proof and begin to recall
the classical Schur's lemma.

\par

\begin{lemma}
\label{lemma:B.3.11}
If $K \in C(\rr{d} \times \rr{d})$,
\begin{equation*}
\sup_{y} \int  |K(x,y)|\, dx \le M
\quad \text{and}\quad
\sup_{x} \int |K(x,y)|\, dy \le M,
\end{equation*}
then the integral operator on $L^2(\rr d)$ with kernel $K$
has norm less than or equal to $M$.
\end{lemma}

\par

For the proof of Theorem \ref{thm:3.6} we also needs the following
version of Cotlar's lemma.

\par

\begin{lemma}
\label{lemma:B.3.12}
Let $x\mapsto T_x$ be a measurable function from $\rr n$ to the set of linear and
continuous operators on $L^2(\rr d)$, and let $h_j(x,y)$, $j=1,2$, be
positive functions on  $\rr{2n}$ such that
\begin{equation}
\label{eq:B.3.63}
\| T_xT_y^* \| \le h_1(x,y)^{2},
\quad
\| T_x^*T_y \| \le h_2(x,y)^{2}.
\end{equation}
If $h_1$ and $h_2$ statisfy
\begin{equation}
\label{eq:B.3.64}
\int h_1(x,y)) \, dx \le M
\quad \text{and}\quad
\int h_2(x,y) \, dx\le M,
\end{equation}
for some constant $M$, then
$$
\Big \Vert \int (T_xf)\, dx \Big \Vert _{L^2}\le M\nm f{L^2},\quad f\in L^2(\rr d).
$$
\end{lemma}

\par

\begin{proof}[Proof of Theorem \ref{thm:3.6}.]
Let $g \in C^\infty(\mathbf R)$ be decreasing and such that 
$g (t) = 1$ for $t < \frac{1}{2}$ and $g (t) = 0$ for $t > 1$,
and set $\chi (x)=g(|x|)$, $x\in \rr d$, and
\[
\psi _{Z}(x,\xi) = \frac{\chi (|x-z|) \chi (|\xi - \zeta |)}
                       {\nm \chi{L^1}^2},\qquad Z=(z,\zeta )\in \rr {2d}.
\]
Then
\begin{gather}
\label{eq:B.3.56}
\supp{\psi_Z} \subseteq U_Z \equiv
                          \sets
                             {(x,\xi) \in \rr{2d}}
                             {|x-z| \le 1,\  |\xi - \zeta | \le 1},
\\
\label{eq:B.3.57}
\max_{|\alpha+\beta| \le N} \;
\sup_{x,\xi \in \rr{d}} \;
\left |
\partial _{\xi}^\alpha \partial ^{\beta} _x \psi _{Z}(x,\xi)
\right | \le C_{N},
\\
\nonumber
\int  \psi _Z(x,\xi) \, dZ= 1,
\end{gather}
where the constants $C_{N}$ are independent of $Z$. For $Z$ fixed,
let
\begin{equation}
\label{eq:B.3.59}
a_{Z}(x,\xi) = \psi_{Z}(x,\xi) a(x,\xi),
\quad \text{and}\quad
A_{Z} = \op_{\varphi}(a_{Z}).
\end{equation}

\par

Now \eqref{eq:B.3.56}, \eqref{eq:B.3.57} and \eqref{eq:B.3.59} imply that
$A_{Z}$ is linear from $C_{0}^\infty(\rr{d})$ to itself, and
$\| A_{Z} f \|_{L^{2}}$ $\le C \| f \|_{L^{2}}$, where the constant $C$ 
is independent of $Z$. In fact, $a_{Z}$ has compact support and 
\eqref{eq:B.3.57} holds. Moreover,
\[
   \psi_{Z} \in C_{0}^\infty \subseteq  \SG^{0,0}_{\min\{r,1\},\min\{\rho,1\}}
\]
and
\[
  A f(x) = \lim_{N \to \infty}
                       \int _{|Z|\le N} 
                       A_{Z}f(x)\,dZ ,
\]
where the limit exists pointwise for all $x \in \rr{d}$ and with 
respect to the strong topology of $L^{2}$.

\par

The result follows if we prove that for all compact sets $K \subset \rr{2d}$
\begin{equation}
\label{eq:B.3.62}
\left \| \int _{K} A_{Z}f \,dZ \right \| _{L^{2}}
\le 
M \| f \|_{L^{2}},
\qquad
f \in C_{0}^\infty(\rr{d}),
\end{equation}
for some constant $M$ independent of $f$ and $K$. To this aim,
we shall prove that $A_Z$ obey the hypothesis in Lemma
\ref{lemma:B.3.12}.
%
%
%
%
%
%

\par

For this reason we consider the kernel $K_{Z_1,Z_2}(x,y)$
of $A_{Z_1}A^*_{Z_2}$, which can be
written as
\begin{equation}
\label{eq:B.3.68}
K_{Z_1,Z_2}(x,y) = (2\pi)^{-d/2}\int e^{i ( \varphi (x,\xi) - \varphi(y,\xi) )} \,
                                        q_{Z_1,Z_2}(x,y,\xi)\,d \xi ,
\end{equation}
with
\[
q_{Z_1,Z_2}(x,y,\xi) = a_{Z_1}(x,\xi) \overline{a_{Z_2}(y,\xi)}\in \cS (\rr{3d})
\]
supported in
$$
\sets { (x,y,\xi)}
                                   {|x-z_1| \le 1,\ |y-z_2 | \le 1,\ 
                                   |\xi - \zeta _1| \le 1,\  
                                   |\xi - \zeta _2| \le 1 }.
$$
We shall prove that $K_{Z_1,Z_2}$ satisfies the
hypotheses of Lemma \ref{lemma:B.3.11} for a suitable $M$.

\par

Let $T$ be the operator
\[
T = H_\fy \cdot ( 1 - L ),
\]
where
\begin{align*}
L &=i \sum_{j=1}^n 
              \left(
                                 \varphi^\prime_{\xi_j}(x,\xi) - \varphi^\prime_{\xi_j}(y,\xi)
                              \right)  \partial_{\xi_j}
\intertext{and}
H_\fy (x,y,\xi ) &=  (1 + \left |
                                 \varphi^\prime_{\xi}(x,\xi) - \varphi^\prime_{\xi}(y,\xi)  
              \right|^{2})^{-1}.
\end{align*}
Then
\begin{eqnarray*}
&
T e^{i ( \varphi(x,\xi) - \varphi(y,\xi) )} =
   e^{i ( \varphi(x,\xi) - \varphi(y,\xi) )},
\end{eqnarray*}
and since
$$
\left|
       \varphi^\prime_\xi(x,\xi) - \varphi^\prime_\xi(y,\xi)    
\right| \gtrsim |x - y|,
$$
by the first part of the proof of Theorem 
\ref{thm:3.4}, we get
$$
H_\fy (x,y,\xi )\lesssim \norm{x-y}^{2}.
$$
Consequently, if $\Div$ is the map $F\mapsto H_\fy \cdot F$, then $L$ and $\Div$
are continuous on $\cS (\rr {3d})$. 
%
%
Since an analogous formula to \eqref{eq:2.2.1} holds
for $({^t T})^N$, by 
the hypotheses and the above observations we have, for arbitrary
$N\in \mathbf N$ and suitable differential operators $V_{N}(\Div  , L)$, 
depending on $\Div$, $L$ and $N$,  
\begin{multline*}
          K_{Z_1,Z_2}(x,y)=
           (2\pi)^{-d/2}\int  T^N \,
                             e^{i ( \varphi(x,\xi) - \varphi(y,\xi) )} \,
                             q_{Z_1,Z_2}(x,y,\xi)\,d \xi 
\\
= (2\pi)^{-d/2}\int  e^{i ( \varphi(x,\xi) - \varphi(y,\xi) )} \,
                             ({^t T})^N \,
                             q_{Z_1,Z_2}(x,y,\xi)\,d \xi
\\
= (2\pi)^{-d/2}\int e^{i ( \varphi(x,\xi) - \varphi(y,\xi) )} \,
                             (\Div ^N + V_{N}(\Div , L) ) \,
                             q_{Z_1,Z_2}(x,y,\xi)\,d \xi
\end{multline*}
Since each term appearing in $V_N(\Div , L)$ contains exactly $N$
operators with $\Div$,
by standard arguments we find
\begin{equation}
\label{eq:B.3.69}
K_{Z_1,Z_2}(x,y) \lesssim 
\tau \left( \frac{\zeta _1-\zeta _2}{2} \right) \,
\tau(x-z_1) \,
\tau(y-z_2) \,
(1 +|x-y|^{2})^{-N},
\end{equation}
where $\tau = \chi_{B_1(0)}$ is the characteristic function of the 
unit ball in $\rr{d}$. Then:
\begin{multline*}
\lefteqn{
\sup _{y} \int  |K_{Z_1,Z_2}(x,y)| \,dx
         }
\\
\lesssim \tau \left( \frac{\zeta _1-\zeta _2}{2} \right) \,
          \sup_{y \in B_1(z_2)} \, 
          \int _{B_1(0)}  (1 + |x + (z_1 - y)|^{2})^{-N} \,dx
\\
\lesssim \tau \left( \frac{\zeta _1-\zeta _2}{2} \right ) \,
          \sup_{y \in B_1(z_2)} \, 
          (1 + |z_1 - y|^{2})^{-N}
\\
\lesssim \tau \left( \frac{\zeta _1-\zeta _2}{2} \right) \,
          (1 + |z_1 - z_2 |^{2})^{-N} 
\end{multline*}
and analogously for $\sup_{x} \int  |K_{Z_1,Z_2}(x,y)|\,dy$, owing to the 
symmetry in the estimate \eqref{eq:B.3.69}. So, all requirements of Lemma
\ref{lemma:B.3.11} are satisfied and summing up, we have:
\begin{eqnarray*}
|\zeta _1-\zeta _2| \ge 2
& \Rightarrow & A_{Z_1} A^*_{Z_2} = 0
\\
|\zeta _1 - \zeta _2| \le 2
& \Rightarrow & \| A_{Z_1} A^*_{Z_2} \| \lesssim
               (1 + |z_1 - z_2|^{2})^{-N}.
\\
\end{eqnarray*}
An analogous estimate can be obtained for $A^*_{Z_1} A_{Z_2}$, in view of  the
symmetry in the role of variables and covariables in $\SG$ phases and amplitudes. Then, also the
requirements \eqref{eq:B.3.63} and \eqref{eq:B.3.64} of Lemma \ref{lemma:B.3.12}
are satisfied. This gives the result.
\end{proof}



\begin{thebibliography}{150}

\bibitem{AndPhD} 
G.D. Andrews,
{\it A Closed Class of $\SG$ Fourier Integral Operators with Applications}.
PhD thesis, Imperial College, London (2004).

\bibitem{AsFu78}
K.~Asada, D.~Fujiwara,
\emph{On Some Oscillatory Transformation in $L^2(\rr{n})$}.
Japan J. Math., \textbf{4} (1978), 229--361.

\bibitem{BaCo} U. Battisti, S. Coriasco, \emph{Wodzicki Residue for Operators on Manifolds with Cylindrical Ends}. {Ann. Global Anal. Geom.} \textbf{40}, 2  (2011), 223-249, DOI 10.1007/s10455-011-9255-3.

\bibitem{Berger}
M.~S. Berger,
\newblock \emph{Nonlinearity and Functional Analysis}.
\newblock Academic Press (1977).

\bibitem{BoBuRo} P. Boggiatto, E. Buzano, L. Rodino, \emph{Global
Hypoellipticity and Spectral Theory}.  Mathematical Research, 92,
Akademie Verlag, Berlin (1996).

\bibitem{Bn1} {J. M. Bony}, \emph{Caract{\'e}risations des
Op{\'e}rateurs Pseudo-Diff{\'e}rentiels.  \rm{In: S{\'e}\-mi\-nai\-re
sur les {\'E}quations aux D{\'e}riv{\'e}es Partielles}}, 1996--1997,
Exp. No. XXIII, S{\'e}min. {\'E}cole Polytech., Palaiseau (1997). 

\bibitem{Bn2} {J. M. Bony}, \emph{Sur l'In\'egalit\'e de
Fefferman-Phong. \rm{In: S\'eminaire sur les  \'Equations aux
D\'eriv\'ees Partielles}}, 1998--1999, Exp. No. III, S\'emin. \'Ecole
Polytech., Palaiseau (1999). 

\bibitem{BC} {J. M. Bony, J. Y. Chemin}, \emph{Espaces Functionnels
Associ\'es au Calcul de Weyl-H{\"o}rmander}. Bull. Soc. math. France
\textbf{122} (1994), 77--118. 

\bibitem{BoL} {J. M. Bony, N. Lerner}, \emph{Quantification
Asymptotique et Microlocalisations d'Ordre Su\-p\'e\-rieur I}.
Ann. Scient. \'Ec. Norm. Sup., \textbf{22} (1989), 377--433. 

\bibitem{Bo11}
M.~Borsero,
\newblock {\em Microlocal Analysis and Spectral Theory
of Elliptic Operators on Non-compact Manifolds.}
\newblock Tesi di Laurea Magistrale in Matematica, Universit{\'a}
di Torino (2011).

\bibitem{BuNi} {E. Buzano, N. Nicola}, \emph{Pseudo-differential
Operators and Schatten-von Neumann Classes. \rm{In: P. Boggiatto,
R. Ashino, M. W. Wong (eds),} Advances in Pseudo-Differential
Operators, Proceedings of the Fourth ISAAC Congress,} Operator
Theory: Advances and Applications, Birkh{\"a}user Verlag, Basel, 
(2004).

\bibitem{BuTo} E. Buzano, J. Toft, \emph{Schatten-von Neumann properties in
the Weyl calculus}. J. Funct. Anal. \textbf{259} (2010), 3080--3114. 

\bibitem{CaRo} 
M. Cappiello, L. Rodino, 
{\em $\SG$-pseudodifferential operators and Gelfand-Shilov spaces}.
{Rocky Mountain J. Math.}, \textbf{36}, 4 (2006), 1117--1148.

\bibitem{CNR07} E.~Cordero, F.~Nicola and L.~Rodino,
 {\it  Boundedness of Fourier Integral Operators 
    on $\mathcal{F} L^p$ spaces}.
    Trans. Amer. Math. Soc.  {\bf 361}  (2009), 6049--6071.    

\bibitem{CorNicRod1} E. Cordero, F. Nicola, L. Rodino, \emph{On the Global
Boundedness of Fourier Integral Operators}. Ann. Global Anal. Geom.
\textbf{38} (2010), 373--398.

\bibitem{Co} H. O. Cordes, \emph{The Technique of Pseudodifferential
Operators.} Cambridge Univ. Press (1995).

\bibitem{coriasco} S. Coriasco, \emph{Fourier Integral Operators in $\SG$
classes I. Composition Theorems and Action on $\SG$ Sobolev spaces}.
Rend. Sem. Mat. Univ. Politec. Torino \textbf{57} (1999), 4, 249-302 (2002).

\bibitem{coriasco2} S. Coriasco, \emph{Fourier Integral Operators
in $\SG$ Classes II: Application to $\SG$ Hyperbolic Cauchy Problems}.
Ann. Univ. Ferrara, Sez. VII -- Sc. Mat. \textbf{44} (1998), 81--122.

\bibitem{CJT1} S. Coriasco, K. Johansson, J. Toft, \emph{Local
wave-front sets of Banach and Fr{\'e}chet types, and pseudo-differential
operators}. {Monatsh. Math.},
\textbf{169} (2013), 285--316.
 
\bibitem{CJT2} S. Coriasco, K. Johansson, J. Toft, \emph{Global wave-front
sets of Banach, Fr{\'e}chet and Modulation space types, and pseudo-differential
operators}. {J. Differential Equations} \textbf{254} (2013), 3228--3258.

\bibitem{CJT3} S. Coriasco, K. Johansson, J. Toft,
\emph{Global Wave-front Sets of Intersection and Union Type.
{\rm
{in: M. Ruzhansky, V. Turunen (eds)}},
Fourier Analysis - Pseudo-differential Operators, Time-Frequency
Analysis and Partial Differential Equations,}
Trends in Mathematics, Birkh{\"a}user,
Heidelberg NewYork Dordrecht London (2014), pp. 91--106.

\bibitem{CJT4} S. Coriasco, K. Johansson, J. Toft,
\emph{Calculus, continuity and global wave-front properties for Fourier integral operators on $\rr{d}$}.
Preprint arXiv:1307.6249 (2013).

\bibitem{CoMa} S. Coriasco, L. Maniccia, \emph{Wave front set
at infinity and hyperbolic linear operators with multiple characteristics}.
{Ann. Global Anal. Geom.} \textbf{24} (2003), 375--400.

\bibitem{CoMa2} S. Coriasco, L. Maniccia, \emph{On the Spectral Asymptotics of Operators on Manifolds with Ends}. {Abstr. Appl. Anal.} vol. 2013, 
Article ID 909782, DOI 10.1155/2013/909782 (2013).

\bibitem{CoPa} S. Coriasco, P. Panarese, \emph{Fourier Integral Operators Defined by 
Classical Symbols with Exit Behaviour}. {Math. Nachr.} \textbf{242} (2002), 61-78.

\bibitem{CoRo} S. Coriasco, L. Rodino, \emph{Cauchy problem for
$\SG$-hyperbolic equations with constant multiplicities}. Ric. di
Matematica, \textbf{48} (Suppl.) (1999), 25--43.

\bibitem{CoRu} S. Coriasco, M. Ruzhansky, \emph{Global $L^p$-continuity
of Fourier Integral Operators}. 
Trans. Amer. Math. Soc. (2014) DOI 10.1090/S0002-9947-2014-05911-4.

\bibitem{CoSch} S. Coriasco, R. Schulz, \emph{Global Wave Front Set of
Tempered Oscillatory Integrals with Inhomogeneous Phase Functions}.
J. Fourier Anal. Appl. \textbf{19}, 5 (2013), 1093-1121
DOI 10.1007/s00041-013-9283-4.

\bibitem{CoSc14} S. Coriasco, R. Schulz, \emph{$\SG$-Lagrangian submanifolds
and their parametrization}. Preprint arXiv:1406.1888 (2014).

\bibitem{CoTo} S. Coriasco, J. Toft, \emph{Asymptotic expansions
for H{\"o}rmander symbol classes in the calculus of pseudo-differential
operators}. J. Pseudo-Differ. Oper. Appl. \textbf{5}, 1 (2014), 27-41, 
DOI 10.1007/s11868-013-0086-9.

\bibitem{ES97}
Y.~V. Egorov, B.-W. Schulze,
\newblock {\em Pseudo-differential operators, singularities, applications}.
  Volume~93 of {\em Operator Theory: Advances and Applications}.
\newblock Birkh\"auser Verlag, Basel (1997).

\bibitem{F1}  H.~G.~Feichtinger, \emph{Modulation spaces on locally
compact abelian groups. Technical report}, {University of
Vienna}, Vienna, 1983; also in: M. Krishna, R. Radha,
S. Thangavelu (Eds) Wavelets and their applications, Allied
Publishers Private Limited, NewDehli Mumbai Kolkata Chennai Hagpur
Ahmedabad Bangalore Hyderbad Lucknow (2003), pp. 99--140.

\bibitem{Feichtinger3}  {H. G. Feichtinger, K. H. Gr{\"o}chenig},
\emph{Banach spaces related to integrable group representations and
their atomic decompositions, I}. J. Funct. Anal. \textbf{86}
(1989), 307--340.

\bibitem{Feichtinger6} {H. G. Feichtinger, K. H. Gr{\"o}chenig},
\emph{Modulation spaces: Looking back and ahead}.
Sampl. Theory Signal Image Process. \textbf{5} (2006), 109--140.

\bibitem{Fo}  {G. B. Folland}, \emph
{Harmonic analysis in phase space}. {Princeton U. P., Princeton}
(1989).

\bibitem{Gro-book} K. Gr\"{o}chenig, \newblock \emph{Foundations of
Time-Frequency Analysis}.
\newblock Birkh\"auser, Boston (2001).

\bibitem{GT} K. Gr{\"o}chenig, J. Toft, \emph{Isomorphism properties of
Toeplitz operators and pseudo-differential operators between modulation
spaces}. J. Anal. Math. \textbf{114} (2011), 255--283.

\bibitem{Ho1} L. H\"ormander, \emph{The Analysis of Linear
Partial Differential Operators}, vol {I--IV}.
Springer-Verlag, Berlin Heidelberg NewYork Tokyo (1983, 1985).

\bibitem{Ku} H.~Kumano-go,
\emph{Pseudo-Differential Operators}.
\newblock MIT Press (1981).

\bibitem{LuRa} F. Luef, Z. Rahbani, \emph{On pseudodifferential operators with
symbols in generalized Shubin classes and an application to Landay-Weyl operators}.
Banach J. Math. Anal. \textbf{5} (2011), 59--72.

\bibitem{MP02}
L.~Maniccia and P.~Panarese,
\newblock \emph{Eigenvalue asymptotics for a class of md-elliptic {$\psi$}do's on
  manifolds with cylindrical exits}.
\newblock {Ann. Mat. Pura Appl.} \textbf{181}, 3 (2002), 283--308, .

%
\bibitem{Me} R. Melrose, \emph{Geometric scattering theory.} Stanford Lectures.
Cambridge University Press, Cambridge (1995).

%
\bibitem{PA72}
C.~Parenti,
\newblock Operatori pseudodifferenziali in $\mathbb{R}^n$ e applicazioni.
\newblock {\em Ann. Mat. Pura Appl.} \textbf{93} (1972) 359--389.

%
\bibitem{RuSu} M. Ruzhansky, M. Sugimoto,  
      \emph{Global $L^2$ boundedness 
      theorems for a class of Fourier integral operators.} 
      Comm. Partial Differential Equations
      \textbf{31} (2006), 547--569. 

%
\bibitem{Sc} E. Schrohe, \emph{Spaces of weighted symbols and
weighted Sobolev spaces on manifolds} In: H. O. Cordes,
B. Gramsch, and H. Widom (eds), Proceedings, Oberwolfach,
\textbf{1256} Springer LMN, New York, (1986), pp. 360-377.
%

     \bibitem{SSS91} A. Seeger, C.D. Sogge, E.M. Stein, 
     \textit{Regularity
       properties of Fourier integral operators.} Ann. of Math.
       \textbf{134} (1991), 231--251.


\bibitem{To8} J. Toft, \emph{Continuity
properties for modulation spaces with applications to
pseudo-differential calculus, II}. {Ann. Global Anal. Geom.}
\textbf{26} (2004), 73--106.

\bibitem{To10} J. Toft, \emph{Schatten-von Neumann properties in the
Weyl calculus, and calculus of metrics on symplectic vector spaces}.
Ann. Glob. Anal. and Geom. \textbf{30} (2006), 169--209.

\end{thebibliography}
\end{document}